\documentclass[a4paper]{amsart}
\usepackage{color}
\usepackage[latin1]{inputenc}
\usepackage[T1]{fontenc}
\usepackage{amsfonts}
\usepackage{amssymb}
\usepackage{amsmath}
\usepackage{amsthm}
\usepackage{graphicx,type1cm,eso-pic,color}
\usepackage{pstricks,pst-plot,pstricks-add}
\usepackage{float}
\newtheorem{theorem}{Theorem}[section]
\newtheorem{lemma}[theorem]{Lemma}
\newtheorem{proposition}[theorem]{Proposition}
\newtheorem{corollary}[theorem]{Corollary}
\newtheorem{definition}[theorem]{Definition}
\newtheorem{assumption}[theorem]{Assumption}
\begin{document}
\setlength\arraycolsep{2pt}
\title{Ergodicity of a Generalized Jacobi Equation and Applications}
\author{Nicolas MARIE}
\address{Laboratoire Modal'X, Universit\'e Paris 10, 92000, Nanterre}
\email{nmarie@u-paris10.fr}
\address{Laboratoire ISTI, ESME Sudria, 75015, Paris}
\email{marie@esme.fr}
\keywords{Euler scheme, Fractional Brownian motion, Jacobi's equation, Malliavin calculus, Morris-Lecar's model, Random dynamical systems, Rough paths, Stochastic differential equations}
\date{}
\maketitle
%


%
\begin{abstract}
Consider a $1$-dimensional centered Gaussian process $W$ with $\alpha$-H\"older continuous paths on the compact intervals of $\mathbb R_+$ ($\alpha\in ]0,1[$) and $W_0 = 0$, and $X$ the local solution in rough paths sense of Jacobi's equation driven by the signal $W$.
\\
The global existence and the uniqueness of the solution are proved via a change of variable taking into account the singularities of the vector field, because it doesn't satisfy the non-explosion condition. The regularity of the associated It\^o map is studied.
\\
By using these deterministic results, Jacobi's equation is studied on probabilistic side : an ergodic theorem in L. Arnold's random dynamical systems framework, and the existence of an explicit density with respect to Lebesgue's measure for each $X_t$, $t > 0$ are proved.
\\
The paper concludes on a generalization of Morris-Lecar's neuron model, where the normalized conductance of the $\textrm{K}^+$ current is the solution of a generalized Jacobi's equation.
\end{abstract}
\tableofcontents
\noindent
\textbf{MSC2010 :} 60H10.
\\
\\
\textbf{Acknowledgements.} Many thanks to Laure Coutin for her advices. This work was supported by ANR Masterie.
%


%
\section{Introduction}
\noindent
Let $W$ be a $1$-dimensional centered Gaussian process with $\alpha$-H\"older continuous paths on the compact intervals of $\mathbb R_+$ ($\alpha\in ]0,1]$) and $W_0 = 0$.
\\
\\
Consider the Jacobi(-type) stochastic differential equation :
\begin{equation}\label{stochastic_auxiliary_equation}
X_t =
x_0
-\int_{0}^{t}
\theta_s\left(X_s - \mu_s\right)ds +
\int_{0}^{t}
\gamma_s\left[\theta_s\left[X_s\left(1 - X_s\right)\right]\right]^{\beta}dW_s
\end{equation}
where, $x_0\in ]0,1[$ is a deterministic initial condition, and the two following assumptions are satisfied :
%


%
\begin{assumption}\label{assumption_beta}
$\beta$ is a deterministic exponent satisfying $\beta\in ]1-\alpha,1[$.
\end{assumption}
%


%
\begin{assumption}\label{auxiliary_assumption}
$\theta$, $\mu$ and $\gamma$ are three continuously differentiable functions on $\mathbb R_+$ such that $\theta_t > 0$, $\mu_t\in ]0,1[$ and $\gamma_t\in\mathbb R$ for every $t\in\mathbb R_+$.
\end{assumption}
\noindent
If the driving signal is a standard Brownian motion, (\ref{stochastic_auxiliary_equation}) taken in the sense of It\^o with $\beta = 1/2$ is the classical Jacobi equation. In that case, the Markov property of the solution is crucial to bypass the difficulties related to the vector field's singularities (cf. S. Karlin and H.M. Taylor \cite{KT81}).
\\
In this paper, deterministic and probabilistic properties of (\ref{stochastic_auxiliary_equation}) are studied by taking it in the sense of rough paths (cf. T. Lyons and Z. Qian \cite{LQ02} and P. Friz and N. Victoir \cite{FV10}). Doss-Sussman's method could also be used since (\ref{stochastic_auxiliary_equation}) is a $1$-dimensional equation (cf. H. Doss \cite{DOSS76} and H.J. Sussman \cite{SUSSMAN78}), but the rough paths theory allows to provide estimates for the $\alpha$-H\"older semi-norm which is more precise than the uniform norm. A priori, even in these frameworks, equation (\ref{stochastic_auxiliary_equation}) admits only a local solution because its vector field is not Lipschitz continuous on neighbourhoods of 0 and 1.
\\
\\
Section 2 is devoted to the global existence and the uniqueness of the solution of Jacobi's equation, the regularity of the associated It\^o map, and a converging approximation scheme with a rate of convergence. Section 3 provides some probabilistic consequences of these deterministic results : the convergence of the approximation scheme mentioned above in $L^p(\Omega;\mathbb P)$ for each $p\geqslant 1$, an ergodic theorem in L. Arnold's random dynamical systems framework, and the existence of an explicit density for $X_t$ with respect to Lebesgue's measure on $(\mathbb R,\mathcal B(\mathbb R))$ for each $t\in\mathbb R_{+}^{*}$. The case of fractional Brownian signals is developed.
\\
Jacobi's equation is tailor-made to model dynamical proportions. For instance, the classical Jacobi equation, taken in the sense of It\^o for $\beta = 1/2$ models the normalized conductance of the $\textrm{K}^+$ current in Morris-Lecar's neuron model provided in S. Ditlevsen and P. Greenwood \cite{DG12}. Section 5 suggests an extension of that neuron model by replacing the classical Jacobi equation by the pathwise generalization studied in this paper.
\\
\\
Some useful results and notations on random dynamical systems (cf. L. Arnold \cite{ARNOLD98}) and Malliavin calculus (cf. D. Nualart \cite{NUALART06}) are stated in Appendix A.
\\
\\
\textbf{Notations.} Consider $t > s\geqslant 0$ and an interval $I\subset\mathbb R$ :
\begin{itemize}
 \item The space $C^0([s,t];I)$ of continuous functions from $[s,t]$ into $I$ is equipped with the uniform norm $\|.\|_{\infty;s,t}$ :
 \begin{displaymath}
 \forall x\in C^0
 \left([s,t];I\right)
 \textrm{, }
 \|x\|_{\infty;s,t} =
 \sup_{u\in [s,t]}
 |x_u|.
 \end{displaymath}
 If $s = 0$, that norm is denoted by $\|.\|_{\infty;t}$.
 \item The space $C^0(\mathbb R_+;I)$ of continuous functions from $\mathbb R_+$ into $I$ is equipped with the compact-open topology. When $I$ is bounded, $C^0(\mathbb R_+;I)$ is sometimes equipped with the uniform norm $\|.\|_{\infty}$ :
 \begin{displaymath}
 \forall x\in C^0\left(\mathbb R_+;I\right)
 \textrm{, }
 \|x\|_{\infty} =
 \sup_{u\in\mathbb R_+}
 |x_u|.
 \end{displaymath}
 \item The space $C^{\alpha}([s,t];I)$ of $\alpha$-H\"older continuous functions $x$ from $[s,t]$ into $I$, such that $x_s = 0$, is equipped with the $\alpha$-H\"older norm $\|.\|_{\alpha;s,t}$ :
 \begin{displaymath}
 \forall x\in C^{\alpha}\left([s,t];I\right)
 \textrm{, }
 \|x\|_{\alpha;s,t} =
 \sup_{s\leqslant v < u\leqslant t}
 \frac{|x_v - x_u|}{|v - u|^{\alpha}}.
 \end{displaymath}
 If $s = 0$, that norm is denoted by $\|.\|_{\alpha;t}$.
 \item The space $C^{\alpha}(\mathbb R_+;I)$ of $I$-valued and $\alpha$-H\"older continuous functions on the compact intervals of $\mathbb R_+$, such that $w_0 = 0$, is equipped with the topology of the convergence on $[0,T]$ for $\|.\|_{\alpha;T}$ and each $T > 0$.
\end{itemize}
%


%
\section{Deterministic properties of Jacobi's equation}
\noindent
Under assumptions \ref{assumption_beta} and \ref{auxiliary_assumption}, the first subsection is devoted to show it admits a unique global solution in the deterministic rough differential equations framework. The regularity of the It\^o map is studied at the second subsection, and a converging approximation scheme is provided at the third one, with a rate of convergence.
\\
\\
Let $w : [0,T]\rightarrow\mathbb R$ be a function satisfying the following assumption :
%


%
\begin{assumption}\label{signal_regularity}
The function $w$ is $\alpha$-H\"older continuous ($\alpha\in ]0,1]$) and $w_0 = 0$.
\end{assumption}
\noindent
Consider the following deterministic analog of equation (\ref{stochastic_auxiliary_equation}) :
\begin{equation}\label{auxiliary_equation}
x_t =
x_0
-\int_{0}^{t}
\theta_s\left(x_s - \mu_s\right)ds +
\int_{0}^{t}
\gamma_s\left[\theta_s\left[x_s\left(1 - x_s\right)\right]\right]^{\beta}dw_s.
\end{equation}
\noindent
The map $A\mapsto [A(1-A)]^{\beta}$ is $C^{\infty}$ and bounded with bounded derivatives on $[\varepsilon,1-\varepsilon]$ for every $\varepsilon > 0$. Then, equation (\ref{auxiliary_equation}) admits a unique solution in the sense of rough paths (cf. \cite{FV10}, Definition 10.17) by applying \cite{FV10}, Exercice 10.56 up to the time
\begin{displaymath}
\tau_{\varepsilon,1-\varepsilon} :=
\inf\left\{
t\in [0,T] :
x_t = \varepsilon
\textrm{ or }
x_t = 1-\varepsilon\right\}
\textrm{ $;$ }
\forall\varepsilon\in ]0,x_0]
\end{displaymath}
with the convention $\inf(\emptyset) = \infty$.
\\
\\
\textbf{Remark.} The underlying, canonical, geometric rough path $\mathbb W$ over $w$ is defined by :
\begin{eqnarray*}
 \mathbb W_{s,t} & := &
 \left(1,\int_{s}^{t}dw_r,\int_{s < r_1 < r_2 < t}dw_{r_1}dw_{r_2},\dots,
 \int_{s < r_1 <\dots < r_{[1/\alpha]} < t}dw_{r_1}\dots dw_{r_{[1/\alpha]}}\right)\\
 & = &
 \left(1,w_t - w_s,\frac{(w_t - w_s)^2}{2},\dots,\frac{(w_t - w_s)^{[1/\alpha]}}{[1/\alpha]!}\right)
 \textrm{ $;$ }\forall t > s\geqslant 0.
\end{eqnarray*}
The purpose of the following subsection is to prove that $\tau_{0,1}\not\in [0,T]$ under assumptions \ref{assumption_beta} and \ref{auxiliary_assumption}, where $\tau_{0,1} > 0$ is defined by $\tau_{\varepsilon,1-\varepsilon}\uparrow\tau_{0,1}$ when $\varepsilon\rightarrow 0$.
\\
\\
\textbf{Remark.} Note that $\tau_{\varepsilon,1-\varepsilon}$ is equal to $\tau_{\varepsilon}\wedge\tau_{1-\varepsilon}$ where,
\begin{displaymath}
\tau_A :=
\inf\left\{t\in [0,T] : x_t = A\right\}
\textrm{ $;$ }
\forall A > 0.
\end{displaymath}
%


%
\subsection{Existence and uniqueness of the solution}
As in N.M. \cite{MARIE12}, the vector field of equation (\ref{auxiliary_equation}) suggests a change of variable which provides a differential equation with additive noise. Under assumptions \ref{assumption_beta} and \ref{auxiliary_assumption}, that new equation allows to show that $\tau_{0,1}\not\in [0,T]$.
\\
\\
Consider the domain
\begin{displaymath}
D = \left\{(u,y)\in [0,1]\times\mathbb R_+ : uy < 1\right\}
\end{displaymath}
and the map $F$ defined on $D$ by
\begin{displaymath}
F(u,y) :=
\int_{0}^{y}
[v(1 - uv)]^{-\beta}dv.
\end{displaymath}
%
%


%
\begin{proposition}\label{properties_F}
Under Assumption \ref{assumption_beta}, the map $F$ satisfies the following properties :
\begin{enumerate}
 \item For every $u\in [0,1]$, the map $F(u,.)$ is strictly increasing on $[0,1/u[$.
 \item For every $u\in ]0,1]$, the map $F(u,.)$ is bijective from $[0,1/u]$ into $[0,F(u,1/u)]$, and its reciprocal map $F_{u}^{-1}$ is continuously derivable on $[0,F(u,1/u)]$.
 \\
 Moreover, $F(0,.)$ is bijective from $[0,\infty[$ into $[0,\infty[$, and its reciprocal map $F_{0}^{-1}$ is continuously derivable on $[0,\infty[$.
 \item For every $y\in\mathbb R_+$, the map $F(.,y)$ is strictly increasing on $[0,1/y[$.
 \item For every $u\in ]0,1]$ and $z\in [0,F(u,1/u)[$, the map $F_{.}^{-1}(z)$ is decreasing on $[0,u]$.
 \item For every $u\in ]0,1]$ and $z\in [0,F(u,1/u)[$,
 \begin{displaymath}
 |\partial_zF_{u}^{-1}(z)|
 \leqslant (1 -\beta)^{\hat\beta}z^{\hat\beta}
 \end{displaymath}
 with $\hat\beta :=\beta/(1 -\beta)$.
 \item Let $G : ]0,1[\rightarrow\mathbb R$ be the map defined by
 \begin{displaymath}
 G(y) :=\theta(\mu - y)\partial_yF(1,y)
 \textrm{ $;$ }
 \forall y\in ]0,1[
 \end{displaymath}
 with $\theta > 0$, $\mu\in ]0,1[$ and $\gamma\in\mathbb R$. There exists $l > 0$ such that :
 \begin{displaymath}
 (G\circ F_{1}^{-1})'(z) < -l
 \textrm{ $;$ }\forall z\in ]0,F(1,1)[.
 \end{displaymath}
 So, $G\circ F_{u}^{-1}$ is strictly decreasing on $]0,F(1,1)[$.
\end{enumerate}
\end{proposition}
%


%
\begin{proof}
\white .\black
\begin{enumerate}
 \item Let $u\in [0,1]$ be arbitrarily chosen. For every $y\in ]0,1/u[$,
 \begin{displaymath}
 \partial_yF(u,y) =
 [y(1 - uy)]^{-\beta} > 0.
 \end{displaymath}
 So, $F(u,.)$ is strictly increasing on $[0,1/u[$.
 \item On the one hand, let $u\in ]0,1]$ be arbitrarily chosen. By Proposition \ref{properties_F}.(1), the map $F(u,.)$ is bijective from $[0,1/u[$ into $[0,F(u,1/u)[$, with
 \begin{displaymath}
 F(u,1/u) = u^{-(1 -\beta)}\int_{0}^{1}[v(1 - v)]^{-\beta}dv.
 \end{displaymath}
 The function $F_{u}^{-1}$ is continuously derivable on $[0,F(u,1/u)[$ because $F(u,.)$ is continuously derivable on $]0,1/u[$, and
 \begin{displaymath}
 \partial_zF_{u}^{-1}(z) =
 [F_{u}^{-1}(z)[1 - uF_{u}^{-1}(z)]]^{\beta}
 \xrightarrow[z\rightarrow 0]{} 0.
 \end{displaymath}
 On the other hand,
 \begin{displaymath}
 F(0,y) =
 \frac{1}{1 -\beta}y^{1-\beta}
 \textrm{ $;$ }
 \forall y\in [0,\infty[
 \end{displaymath}
 and
 \begin{displaymath}
 F_{0}^{-1}(z) =
 (1 -\beta)^{1/(1-\beta)}
 z^{1/(1-\beta)}
 \textrm{ $;$ }
 \forall z\in [0,\infty[.
 \end{displaymath}
 \item Let $y\in\mathbb R_+$ be arbitrarily chosen. For every $u\in ]0,1/y[$,
 \begin{displaymath}
 \partial_uF(u,y) =
 \beta\int_{0}^{y}
 v^2[v(1 - uv)]^{-(\beta + 1)}dv > 0.
 \end{displaymath}
 So, $F(.,y)$ is strictly increasing on $[0,1/y[$.
 \item Let $u\in ]0,1]$, $z\in [0,F(u,1/u)[$ and $u_1,u_2\in [0,u]$ be arbitrarily chosen. Assume that $u_1 > u_2$. Since $F(u_1,.)$ (resp. $F(u_2,.)$) is bijective from $[0,1/u_1[$ into $[0,F(u_1,1/u_1)[\supset [0,F(u,1/u)[$ (resp. $[0,1/u_2[$ into $[0,F(u_2,1/u_2)[\supset [0,F(u,1/u)[$) :
 \begin{displaymath}
 \exists (y_1,y_2)\in [0,1/u_1[\times [0,1/u_2[ :
 F(u_1,y_1) = F(u_2,y_2) = z.
 \end{displaymath}
 So, $F_{u_1}^{-1}(z) = y_1$ and $F_{u_2}^{-1}(z) = y_2$. Suppose that $y_1 > y_2$. Since $F(.,y_2)$ and $F(u_1,.)$ are strictly increasing on $[0,1/y_2[$ and $[0,1/u_1[$ respectively :
 \begin{displaymath}
 z = F(u_2,y_2) < F(u_1,y_2) < F(u_1,y_1) = z.
 \end{displaymath}
 There is a contradiction, so $y_1 = F_{u_1}^{-1}(z) < F_{u_2}^{-1}(z) = y_2$. Therefore, $F_{.}^{-1}(z)$ is decreasing on $[0,u]$.
 \item Let $u\in ]0,1]$ be arbitrarily chosen. As shown previously :
 \begin{displaymath}
 \partial_zF_{u}^{-1}(z) =
 [F_{u}^{-1}(z)[1 - uF_{u}^{-1}(z)]]^{\beta}
 \textrm{ $;$ }\forall z\in [0,F(u,1/u)[.
 \end{displaymath}
 Let $z\in [0,F(u,1/u)[$ be arbitrarily chosen. On the one hand, since $F_{.}^{-1}(z)$ is decreasing on $[0,u]$ by Proposition \ref{properties_F}.(4) :
 \begin{displaymath}
 0\leqslant F_{u}^{-1}(z)\leqslant
 F_{0}^{-1}(z) =
 [(1 - \beta)z]^{1/(1-\beta)}.
 \end{displaymath}
 On the other hand, since $(u,F_{u}^{-1}(z))\in D$ :
 \begin{displaymath}
 0\leqslant 1 - uF_{u}^{-1}(z)\leqslant 1.
 \end{displaymath}
 Therefore,
 \begin{displaymath}
 |\partial_zF_{u}^{-1}(z)|
 \leqslant (1 -\beta)^{\hat\beta}z^{\hat\beta}.
 \end{displaymath}
 \item Let $F_1$ be the function defined by $F_1(y) := F(1,y)$ for every $y\in [0,1]$. On $]0,F_1(1)[$ :
 \begin{displaymath}
 (G\circ F_{1}^{-1})' =
 \frac{G'\circ F_{1}^{-1}}{F_1'\circ F_{1}^{-1}}.
 \end{displaymath}
 Since $F_{1}^{-1}$ is a $]0,1[$-valued map on $]0,F_1(1)[$, it is sufficient to show that $-G'/F_1'$ is $\mathbb R_{+}^{*}$-valued on $]0,1[$ in order to show that $(G\circ F_{1}^{-1})'$ is $\mathbb R_{-}^{*}$-valued on $]0,F_1(1)[$. For every $y\in ]0,1[$,
 \begin{displaymath}
 -\frac{G'(y)}{F_1'(y)} =
 \theta\left[
 \beta
 \frac{(\mu - y)(1 - 2y)}{y(1-y)} + 1\right].
 \end{displaymath}
 Then, $-G'(y)/F_1'(y) > 0$ if and only if $P(y) > 0$, where
 \begin{eqnarray*}
  P(y) & := &
  (2\beta - 1)y^2 +
  (1-\beta - 2\mu\beta)y +
  \mu\beta\\
  & = &
  (2\beta - 1)\left[
  y + \frac{1-\beta -2\mu\beta}{2(2\beta - 1)}\right]^2 -
  \frac{(1-\beta -2\mu\beta)^2}{4(2\beta - 1)} +\mu\beta.
 \end{eqnarray*}
 On the one hand, $P(0) = \mu\beta > 0$ and $P(1) = \beta(1-\mu) > 0$. Then, for $\beta < 1/2$, $P(y) > P(0)\wedge P(1) > 0$ for every $y\in ]0,1[$.
 \\
 \\
 On the other hand, assume that $\beta > 1/2$ and consider
 \begin{displaymath}
 y^*(\beta,\mu) :=
 -\frac{1-\beta -2\mu\beta}{2(2\beta - 1)}
 \textrm{ and }
 \varphi(\beta,\mu) :=
 -\frac{(1-\beta -2\mu\beta)^2}{4(2\beta - 1)} +
 \mu\beta.
 \end{displaymath}
 If $y^*(\beta,\mu)\not\in ]0,1[$, then $P(y) > P(0)\wedge P(1) > 0$.
 \\
 \\
 If $y^*(\beta,\mu)\in ]0,1[$, then
 \begin{displaymath}
 -\frac{(1-\beta -2\mu\beta)^2}{4(2\beta - 1)} >
 \frac{1-\beta}{2} -\mu\beta.
 \end{displaymath}
 So, $P(y) > P[y^*(\beta,\mu)] = \varphi(\beta,\mu) > (1-\beta)/2 > 0$.
 \\
 \\
 In conclusion, there exists $l > 0$ such that $(G\circ F_{1}^{-1})'(z) < -l$ for every $z\in ]0,F_1(1)[$, because $G'(y)/F_1'(y) < 0$ for every $y\in ]0,1[$ and
 \begin{displaymath}
 \lim_{y\rightarrow 0^+}
 \frac{G'(y)}{F_1'(y)} =
 \lim_{y\rightarrow 1^-}
 \frac{G'(y)}{F_1'(y)} = -\infty.
 \end{displaymath}
\end{enumerate}
\end{proof}
\noindent
On the two following figures, $F$, $F(u,.)$ and $F_{u}^{-1}$ are plotted for several values of $u\in [0,1]$, $y\in ]0,1[$ and $\beta\in ]0,1[$. It is sufficient in order to illustrate the properties of $F$ stated at Proposition \ref{properties_F} :
\begin{figure}[H]
\centering
\includegraphics[scale = 0.45]{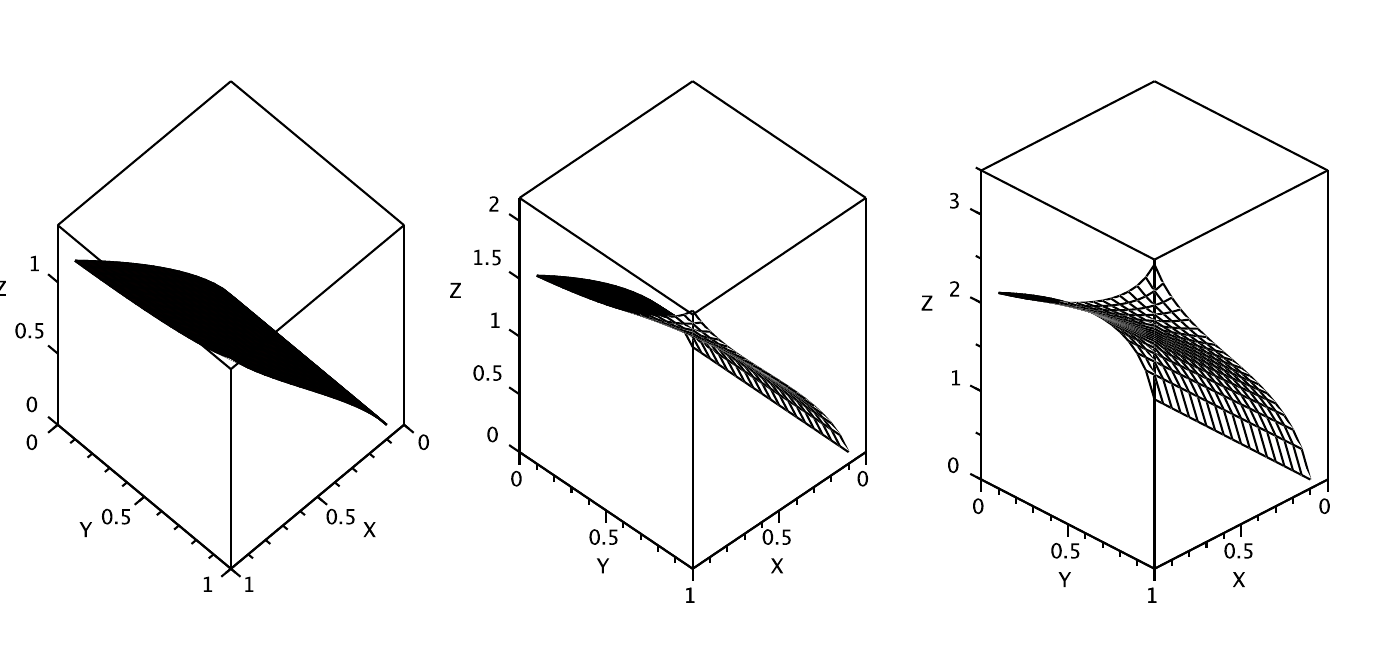}
\caption{Plots of $F$ on $[0,1]\times]0,1[$ for $\beta = 0.25, 0.5, 0.75$}
\end{figure}
\begin{figure}[H]
\centering
\includegraphics[scale = 0.45]{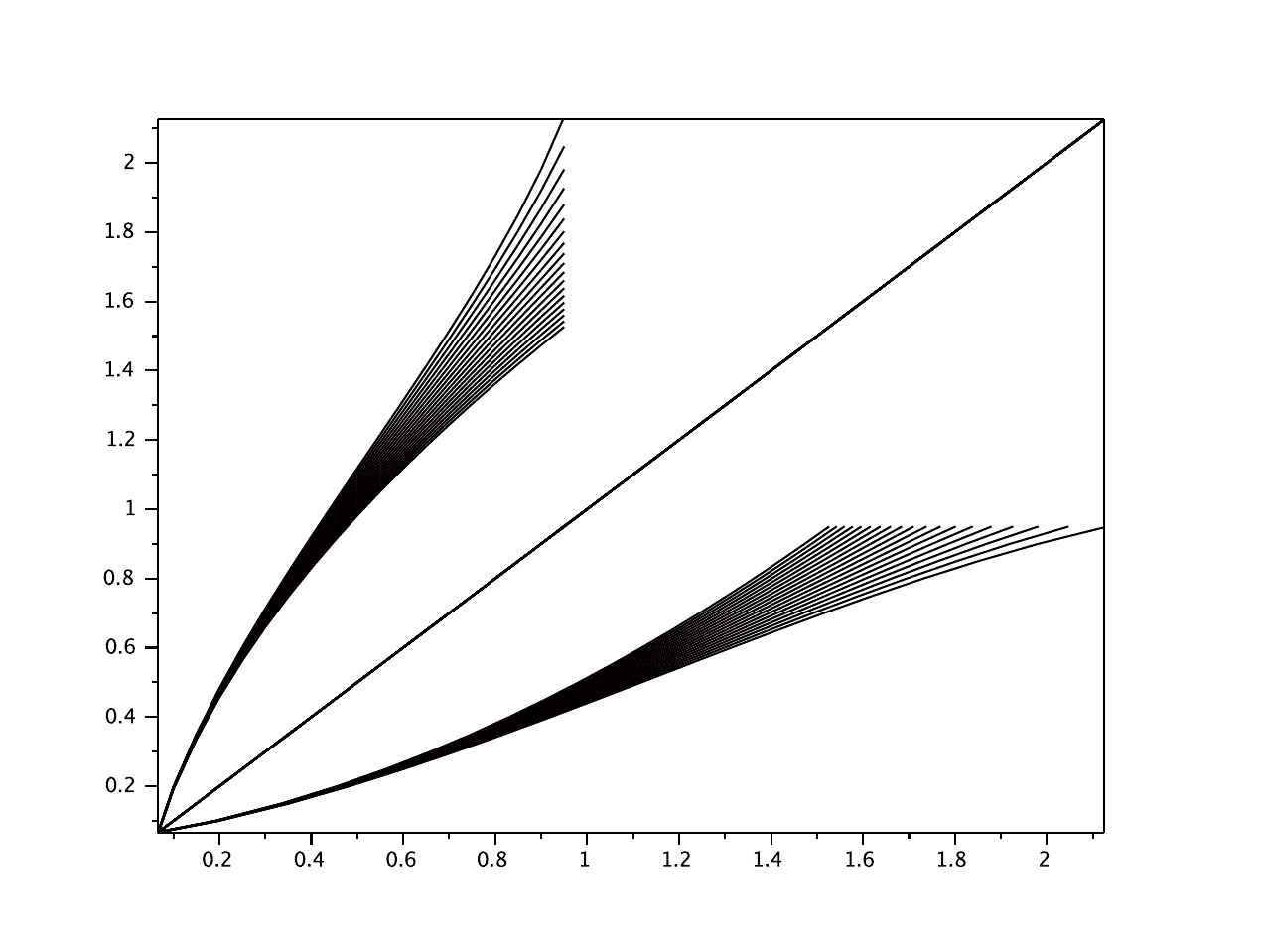}
\caption{Plots of $F(u,.)$ and $F_{u}^{-1}$ for $u\in [0,1]$ and $\beta = 0.5$}
\end{figure}
\noindent
By using the change of variable
\begin{displaymath}
\tilde x_t :=
F\left(e^{-\Theta_t},
e^{\Theta_t}x_t\right)
\textrm{ with }
\Theta_t :=
\int_{0}^{t}\theta_sds
\textrm{ $;$ }\forall t\in [0,\tau_{0,1}],
\end{displaymath}
the following theorem shows that $\tau_{0,1}\notin [0,T]$ :
%


%
\begin{theorem}\label{existence_uniqueness_J}
Under assumptions \ref{assumption_beta}, \ref{auxiliary_assumption} and \ref{signal_regularity}, with initial condition $x_0\in ]0,1[$, equation (\ref{auxiliary_equation}) admits a unique solution $\pi(0,x_0;w)$ on $[0,T]$.
\end{theorem}
%


%
\begin{proof}
For $x_0\in ]0,1[$ and $\varepsilon\in ]0,x_0]$, let $x$ be the solution of equation (\ref{auxiliary_equation}) on $[0,\tau_{\varepsilon,1-\varepsilon}]$ with initial condition $x_0$. Then, $(e^{-\Theta_t},e^{\Theta_t}x_t)\in D$ for every $t\in [0,\tau_{\varepsilon,1-\varepsilon}]$. By applying the change of variable formula (cf. \cite{LQ02}, Theorem 5.4.1) to $(e^{-\Theta},e^{\Theta}x)$ and to the map $F$ between $0$ and $t\in [0,\tau_{\varepsilon,1-\varepsilon}]$ :
\begin{eqnarray*}
 \tilde x_t - \tilde x_0 & = &
 \int_{0}^{t}
 \partial_u F\left(e^{-\Theta_s},e^{\Theta_s}x_s\right)de^{-\Theta_s} +
 \int_{0}^{t}
 \partial_y F\left(e^{-\Theta_s},e^{\Theta_s}x_s\right)d\left(e^{\Theta}x\right)_s\\
 & = &
 -\beta\int_{0}^{t}
 \theta_se^{-\Theta_s}
 \int_{0}^{F_{e^{-\Theta_s}}^{-1}(\tilde x_s)}
 v^2\left[v\left(1 - e^{-\Theta_s}v\right)\right]^{-(\beta + 1)}dvds +\\
 & &
 \int_{0}^{t}
 \partial_y F\left(e^{-\Theta_s},e^{\Theta_s}x_s\right)d\left(e^{\Theta}x\right)_s
\end{eqnarray*}
with $\tilde x_0 := F(1,x_0)$. Moreover,
\begin{eqnarray*}
 \int_{0}^{t}
 \partial_y F\left(e^{-\Theta_s},e^{\Theta_s}x_s\right)d\left(e^{\Theta}x\right)_s
 & = &
 \int_{0}^{t}
 \partial_y F\left(e^{-\Theta_s},e^{\Theta_s}x_s\right)
 \left(\theta_se^{\Theta_s}x_sds + e^{\Theta_s}dx_s\right)\\
 & = &
 w_{t}^{\vartheta} +
 \int_{0}^{t}
 \partial_y F\left(e^{-\Theta_s},e^{\Theta_s}x_s\right)
 \mu_s\theta_se^{\Theta_s}ds\\
 & = &
 w_{t}^{\vartheta} +
 \int_{0}^{t}
 \frac{\mu_s\theta_se^{\Theta_s}}
 {\partial_y F_{e^{-\Theta_s}}^{-1}\left(\tilde x_s\right)}ds
\end{eqnarray*}
where,
\begin{displaymath}
w_{t}^{\vartheta} :=
\int_{0}^{t}\vartheta_sdw_s
\textrm{ with }
\vartheta_t :=
\theta_{t}^{\beta}\gamma_t
e^{(1 -\beta)\Theta_t}.
\end{displaymath}
Then, $\tilde x$ is the solution of the following differential equation with additive \mbox{noise $w^{\vartheta}$ :}
\begin{eqnarray}
 \label{auxiliary_equation_tilde_x}
 \tilde x_t - \tilde x_0 & = &
 -\beta\int_{0}^{t}
 \theta_se^{-\Theta_s}
 \int_{0}^{F_{e^{-\Theta_s}}^{-1}(\tilde x_s)}
 v^2\left[v\left(1 - e^{-\Theta_s}v\right)\right]^{-(\beta + 1)}dvds +\\
 & &
 \int_{0}^{t}
 \frac{\mu_s\theta_se^{\Theta_s}}
 {\partial_y F_{e^{-\Theta_s}}^{-1}\left(\tilde x_s\right)}ds +
 w_{t}^{\vartheta}
 \nonumber
\end{eqnarray}
for every $t\in [0,\tau_{\varepsilon,1-\varepsilon}]$. When $\varepsilon\rightarrow 0$ :
\begin{itemize}
 \item If $\tau_{0,1} = \tau_0$, for $t\in [0,\tau_0[$ :
 \begin{eqnarray}
  \label{solving_decomposition}
  \tilde x_t +
  \int_{t}^{\tau_0}
  \frac{\mu_s\theta_se^{\Theta_s}}
  {\partial_y F_{e^{-\Theta_s}}^{-1}\left(\tilde x_s\right)}ds
  & = & \\
  w_{t}^{\vartheta} - w_{\tau_0}^{\vartheta} +
  \beta\int_{t}^{\tau_0}
  \theta_se^{-\Theta_s}
  \int_{0}^{F_{e^{-\Theta_s}}^{-1}(\tilde x_s)}
  v^2\left[v\left(1 - e^{-\Theta_s}v\right)\right]^{-(\beta + 1)}dvds. & &
  \nonumber
 \end{eqnarray}
 Since $w^{\vartheta} : [0,T]\rightarrow\mathbb R$ is $\alpha$-H\"older continuous, the right-hand side of (\ref{solving_decomposition}) is less or equal than $C(\tau_0 - t)^{\alpha}$ with
 \begin{displaymath}
 C :=
 \|w^{\vartheta}\|_{\alpha;T} +
 \beta\|\theta\|_{\infty;T}
 T^{1-\alpha}
 \int_{0}^{1}
 v^2[v(1 - v)]^{-(\beta + 1)}dv.
 \end{displaymath}
 The two terms of the sum of the left-hand side in equation (\ref{solving_decomposition}) are positive, then $\tilde x_s\leqslant C(\tau_0 - s)^{\alpha}$ for every $s\in [0,\tau_0[$, and by Proposition \ref{properties_F}.(5) :
 \begin{eqnarray*}
  \partial_y F_{e^{-\Theta_s}}^{-1}(\tilde x_s)
  & \leqslant &
  \partial_y F_{0}^{-1}(\tilde x_s) =
  (1-\beta)^{\hat\beta}\tilde x_{s}^{\hat\beta}\\
  & \leqslant &
  (1-\beta)^{\hat\beta}
  C^{\hat\beta}
  (\tau_0 - s)^{\alpha\hat\beta}.
 \end{eqnarray*}
 Therefore,
 \begin{displaymath}
 (1-\beta)^{-\hat\beta}C^{-\hat\beta}
 \left(
 \min_{s\in [0,T]}\mu_s\theta_s\right)
 \int_{t}^{\tau_0}
 (\tau_0 - s)^{-\alpha\hat\beta}ds
 \leqslant
 C(\tau_0 - t)^{\alpha}.
 \end{displaymath}
 Under Assumption \ref{assumption_beta}, the previous inequality gives $\tau_0\not\in [0,T]$.
 \item If $\tau_{0,1} = \tau_1$, consider $\hat x_t := 1 - x_t$ for each $t\in [0,\tau_1[$. The function $\hat x$ satisfies
 \begin{displaymath}
 d\hat x_t =
 -\theta_t(\hat x_t - \hat\mu_t)dt +
 \hat\gamma_t[\theta_t[\hat x_t(1-\hat x_t)]]^{\beta}dw_t
 \end{displaymath}
 with $\hat\mu_t := 1 - \mu_t$ and $\hat\gamma_t := -\gamma_t$.
 \\
 In other words, $\hat x$ is the solution of equation (\ref{auxiliary_equation}) with these new coefficients, also satisfying Assumption \ref{auxiliary_assumption}. Then, under Assumption \ref{assumption_beta} :
 \begin{displaymath}
 \tau_1 =
 \inf\left\{t\in [0,T] : \hat x_t = 0\right\}
 \not\in [0,T].
 \end{displaymath}
\end{itemize}
The solution $x$ doesn't hit $0$ or $1$ on $[0,T]$ because $\tau_{0,1}\not\in [0,T]$. Therefore, equation (\ref{auxiliary_equation}) admits a unique $]0,1[$-valued solution $\pi(0,x_0;w)$ on $[0,T]$.
\end{proof}
\noindent
Let $w : \mathbb R_+\rightarrow\mathbb R$ be a function satisfying the following assumption :
%


%
\begin{assumption}\label{signal_regularity_R_+}
The function $w$ is $\alpha$-H\"older continuous on the compact intervals of $\mathbb R_+$ ($\alpha\in ]0,1]$) and $w_0 = 0$.
\end{assumption}
%


%
\begin{corollary}\label{existence_uniqueness_J_R_+}
Under assumptions \ref{assumption_beta}, \ref{auxiliary_assumption} and \ref{signal_regularity_R_+}, equation (\ref{auxiliary_equation}) admits a unique solution, $\alpha$-H\"older continuous on the compact intervals of $\mathbb R_+$, and still denoted by $\pi(0,x_0;w)$.
\end{corollary}
%


%
\begin{proof}
By Theorem \ref{existence_uniqueness_J}, equation (\ref{auxiliary_equation}) admits a unique solution on $\mathbb R_+$ by putting
\begin{displaymath}
\pi(0,x_0;w)|_{[0,T]} :=
\pi(0,x_0;w|_{[0,T]})
\end{displaymath}
for each $T > 0$. Since $\pi(0,x_0;w|_{[0,T]})$ is $\alpha$-H\"older continuous on $[0,T]$ for every $T > 0$, $\pi(0,x_0;w)$ is $\alpha$-H\"older continuous on the compact intervals of $\mathbb R_+$ by construction.
\end{proof}
%


%
\subsection{Regularity of the It\^o map}
In a first part, propositions \ref{Ito_map_continuity} and \ref{continuous_differentiability} extend the existing regularity results for the It\^o map (cf. \cite{FV10}, chapters 10 and 11) to equation (\ref{auxiliary_equation}), which has a singular vector field. Moreover, at Proposition \ref{Ito_map_continuity}, it is proved that $\pi(0,x_0;.)$ is (globally) Lipschitz continuous.
\\
In a second part, still by using the particular form of the vector field of equation (\ref{auxiliary_equation}), corollaries \ref{monotonicity_initial_condition}, \ref{extension_by_continuity_y} and \ref{Ito_map_Lipschitz} provide some properties of $\pi(0,.;w)$ that will be essential to study the ergodicity of the process $X$ at Subsection 3.1.
\\
\\
In the sequel, the parameters $\theta$, $\mu$ and $\gamma$ are constant, and satisfy the following assumption :
%


%
\begin{assumption}\label{auxiliary_assumption_cst}
$\theta$, $\mu$ and $\gamma$ are three deterministic constants such that $\theta > 0$, $\mu\in ]0,1[$ and $\gamma\in\mathbb R$.
\end{assumption}
\noindent
In this subsection, the results on equation (\ref{auxiliary_equation}) are obtained via the simple change of variable $y_t := F(x_t)$, where $F(x) := F(1,x)$ for every $x\in [0,1]$.
\\
By applying the change of variable formula (cf. \cite{LQ02}, Theorem 5.4.1) to the function $x$ and to the map $F$ between $0$ and $t\in\mathbb R_+$ :
\begin{equation}\label{auxiliary_equation_y}
y_t =
y_0 +
\int_{0}^{t}
(G\circ F^{-1})(y_s)ds +
\gamma\theta^{\beta}w_t.
\end{equation}
with $G(x) := \theta(\mu - x)F'(x)$ for every $x\in ]0,1[$.
%


%
\begin{proposition}\label{Ito_map_continuity}
Under assumptions \ref{assumption_beta}, \ref{signal_regularity_R_+} and \ref{auxiliary_assumption_cst}, the It\^o map $\pi(0,.)$ is continuous from
\begin{displaymath}
]0,1[\times C^{\alpha}
\left(\mathbb R_+;\mathbb R\right)
\textrm{ into }
C^0\left(\mathbb R_+;]0,1[\right).
\end{displaymath}
Moreover, for every $T > 0$, $0 < x_{0}^{1}\leqslant x_{0}^{2} < 1$ and $w^1,w^2\in C^{\alpha}([0,T];\mathbb R)$,
\begin{displaymath}
\|\pi(0,x_{0}^{1} ; w^1) -\pi(0,x_{0}^{2} ; w^2)\|_{\infty ; T}
\leqslant
C_T(x_{0}^{1},x_{0}^{2})(|x_{0}^{1} - x_{0}^{2}| +
\|w^1 - w^2\|_{\alpha ; T})
\end{displaymath}
with $C_T(x_{0}^{1},x_{0}^{2}) := \|(F^{-1})'\|_{\infty ; [0,F(1)]}[\|F'\|_{\infty ; [x_{0}^{1},x_{0}^{2}]}\vee (2T^{\alpha}\gamma\theta^{\beta})]$.
\end{proposition}
%


%
\begin{proof}
For $i = 1,2$, consider $x_{0}^{i}\in ]0,1[$ and $w^i :\mathbb R_+\rightarrow\mathbb R$ a function satisfying Assumption \ref{signal_regularity_R_+}. Under assumptions \ref{assumption_beta} and \ref{auxiliary_assumption_cst}, let $x^i$ be the solution of equation (\ref{auxiliary_equation}) with initial condition $x_{0}^{i}$ and signal $w^i$, and put $y^i := F(x_{.}^{i})$.
\\
\\
The first step shows that the It\^o map associated to $y^1$ and $y^2$ is Lipschitz continuous from
\begin{displaymath}
]0,F(1)[\times C^{\alpha}
\left([0,T];\mathbb R\right)
\textrm{ into }
C^0\left([0,T];]0,F(1)[\right).
\end{displaymath}
At the second step, the expected results on $\pi(0,.)$ are deduced from the first one.
\\
\\
Let $T > 0$ be arbitrarily chosen.
\\
\\
\textbf{Step 1.} On the one hand, consider $t\in [0,\tau_{\textrm{cross}}]$ where
\begin{displaymath}
\tau_{\textrm{cross}} :=
\inf\left\{s\in [0,T] : y_{s}^{1} = y_{s}^{2}\right\},
\end{displaymath}
and suppose that $y_{0}^{1}\geqslant y_{0}^{2}$.
\\
\\
Since $y^1$ and $y^2$ are continuous on $[0,T]$ by construction, for every $s\in [0,\tau_{\textrm{cross}}]$, $y_{s}^{1}\geqslant y_{s}^{2}$ and then,
\begin{displaymath}
(G\circ F^{-1})(y_{s}^{1}) -
(G\circ F^{-1})(y_{s}^{2})\leqslant 0
\end{displaymath}
because $G\circ F^{-1}$ is a decreasing map (cf. Proposition \ref{properties_F}.(6)). Therefore,
\begin{eqnarray*}
 |y_{t}^{1} - y_{t}^{2}| & = &
 y_{t}^{1} - y_{t}^{2}\\
 & = &
 y_{0}^{1} - y_{0}^{2} +
 \int_{0}^{t}
 [(G\circ F^{-1})(y_{s}^{1}) -
 (G\circ F^{-1})(y_{s}^{2})]ds +
 \gamma\theta^{\beta}(w_{t}^{1} - w_{t}^{2})\\
 & \leqslant &
 |y_{0}^{1} - y_{0}^{2}| +
 \gamma\theta^{\beta}
 \|w^1 - w^2\|_{\infty;T}.
\end{eqnarray*}
Symmetrically, one can show that this inequality is still true when $y_{0}^{1}\leqslant y_{0}^{2}$.
\\
\\
On the other hand, consider $t\in [\tau_{\textrm{cross}},T]$,
\begin{displaymath}
\tau_{\textrm{cross}}(t) :=
\sup\left\{s\in\left[\tau_{\textrm{cross}},t\right] : y_{s}^{1} = y_{s}^{2}\right\}
\end{displaymath}
and suppose that $y_{t}^{1}\geqslant y_{t}^{2}$.
\\
\\
Since $y^1$ and $y^2$ are continuous on $[0,T]$ by construction, for every $s\in [\tau_{\textrm{cross}}(t),t]$, $y_{s}^{1}\geqslant y_{s}^{2}$ and then,
\begin{displaymath}
(G\circ F^{-1})(y_{s}^{1}) -
(G\circ F^{-1})(y_{s}^{2})\leqslant 0
\end{displaymath}
because $G\circ F^{-1}$ is a decreasing map. Therefore,
\begin{eqnarray*}
 |y_{t}^{1} - y_{t}^{2}| & = &
 y_{t}^{1} - y_{t}^{2}\\
 & = &
 \int_{\tau_{\textrm{cross}}(t)}^{t}
 [(G\circ F^{-1})(y_{s}^{1}) -
 (G\circ F^{-1})(y_{s}^{2})]ds +\\
 & &
 \gamma\theta^{\beta}
 (w_{t}^{1} - w_{t}^{2}) -
 \gamma\theta^{\beta}
 [w_{\tau_{\textrm{cross}}(t)}^{1} -
 w_{\tau_{\textrm{cross}}(t)}^{2}]\\
 & \leqslant &
 2\gamma\theta^{\beta}\|w^1 - w^2\|_{\infty;T}.
\end{eqnarray*}
Symmetrically, one can show that this inequality is still true when $y_{t}^{1}\leqslant y_{t}^{2}$.
\\
\\
By putting these cases together and since the obtained upper-bounds are not depending on $t$ :
\begin{equation}\label{Lipschitz_auxiliary_y}
\|y^1 - y^2\|_{\infty;T}
\leqslant
|y_{0}^{1} - y_{0}^{2}| +
2T^{\alpha}\gamma\theta^{\beta}\|w^1 - w^2\|_{\alpha;T}.
\end{equation}
Then, the It\^o map associated to $y^1$ and $y^2$ is Lipschitz continuous from
\begin{displaymath}
]0,F(1)[\times C^{\alpha}
\left([0,T];\mathbb R\right)
\textrm{ into }
C^0\left([0,T];]0,F(1)[\right).
\end{displaymath}
\textbf{Step 2.} Since $x^i = F^{-1}(y_{.}^{i})$, $F^{-1}$ is continuously differentiable from $[0,F(1)]$ into $[0,1]$, and $F$ is continuously differentiable from $]0,1[$ into $]0,F(1)[$, by inequality (\ref{Lipschitz_auxiliary_y}) :
\begin{eqnarray*}
 \|\pi(0,x_{0}^{1} ; w^1) -\pi(0,x_{0}^{2} ; w^2)\|_{\infty ; T}
 & \leqslant &
 \|(F^{-1})'\|_{\infty;[0,F(1)]}\times\\
 & &
 [|F(x_{0}^{1}) - F(x_{0}^{2})| +
 2T^{\alpha}\gamma\theta^{\beta}\|w^1 - w^2\|_{\alpha ; T}]\\
 & \leqslant &
 C_T(x_{0}^{1},x_{0}^{2})
 (|x_{0}^{1} - x_{0}^{2}| +
 \|w^1 - w^2\|_{\alpha ; T}).
\end{eqnarray*}
So, $\pi(0,.)$ is locally Lipschitz continuous from
\begin{displaymath}
]0,1[\times C^{\alpha}
\left([0,T];\mathbb R\right)
\textrm{ into }
C^0\left([0,T];]0,1[\right).
\end{displaymath}
Consider $w\in C^{\alpha}(\mathbb R_+;\mathbb R)$ and a sequence $(w^n,n\in\mathbb N)$ of elements of $C^{\alpha}(\mathbb R_+;\mathbb R)$ such that :
\begin{displaymath}
\forall T > 0\textrm{, }
\lim_{n\rightarrow\infty}
\left\|w^n|_{[0,T]} - w|_{[0,T]}\right\|_{\alpha;T} = 0.
\end{displaymath}
For each $T > 0$ and every $x_0\in ]0,1[$,
\begin{eqnarray*}
 \lim_{n\rightarrow\infty}
 \left\|\pi(0,x_0;w^n)|_{[0,T]} -
 \pi(0,x_0;w)|_{[0,T]}\right\|_{\infty;T} & = &\\
 \lim_{n\rightarrow\infty}
 \left\|\pi(0,x_0;w^n|_{[0,T]}) -
 \pi(0,x_0;w|_{[0,T]})\right\|_{\infty;T} & = & 0
\end{eqnarray*}
because $\pi(0,.)$ is continuous from $]0,1[\times C^{\alpha}([0,T];\mathbb R)$ into $C^0([0,T];]0,1[)$.
\\
\\
That achieves the proof.
\end{proof}
\noindent
Let us now show that the continuous differentiability of the It\^o map established at \cite{FV10}, Theorem 11.3 extends to equation (\ref{stochastic_auxiliary_equation}) :
%


%
\begin{proposition}\label{continuous_differentiability}
Under assumptions \ref{assumption_beta}, \ref{signal_regularity_R_+} and \ref{auxiliary_assumption_cst}, the It\^o map $\pi(0,.)$ is continuously differentiable from
\begin{displaymath}
]0,1[\times C^{\alpha}
\left([0,T];\mathbb R\right)
\textrm{ into }
C^0\left([0,T];]0,1[\right)
\end{displaymath}
for each $T > 0$.
\end{proposition}
%


%
\begin{proof}
For the sake of readability, the space $]0,1[\times C^{\alpha}([0,T];\mathbb R)$ is denoted by $E$.
\\
\\
Consider $(x_{0}^{0},w^0)\in E$, $x^0 := \pi(0,x_{0}^{0};w^0)$,
\begin{displaymath}
m_0\in
\left]0,
\left(\min_{t\in [0,T]}x_{t}^{0}\right)\wedge
\left(1 -\max_{t\in [0,T]}x_{t}^{0}\right)\right[
\end{displaymath}
and
\begin{displaymath}
\varepsilon_0 :=
\left(-m_0 + \min_{t\in [0,T]}x_{t}^{0}\right)\wedge
\left(1 - m_0 -\max_{t\in [0,T]}x_{t}^{0}\right).
\end{displaymath}
Since $\pi(0,.)$ is continuous from $E$ into $C^0([0,T];\mathbb R)$ by Proposition \ref{Ito_map_continuity} :
\begin{eqnarray}
 \forall\varepsilon\in ]0,\varepsilon_0]
 \textrm{, }
 \exists\eta > 0 & : &
 \forall (x_0,w)\in E,
 \nonumber\\
 \label{continuity_z}
 (x_0,w) & \in & B_E((x_{0}^{0},w^0);\eta)
 \Longrightarrow
 \|\pi(0,x_0;w) - x^0\|_{\infty;T} < \varepsilon\leqslant\varepsilon_0.
\end{eqnarray}
In particular, for every $(x_0,w)\in B_E((x_{0}^{0},w^0);\eta)$, the function $\pi(0,x_0;w)$ is $[m_0,1 - m_0]$-valued and $[m_0,1 - m_0]\subset ]0,1[$.
\\
\\
In \cite{FV10}, the continuous differentiability of the It\^o map with respect to the initial condition and the driving signal is established at theorems 11.3 and 11.6. In order to derive the It\^o map with respect to the driving signal at point $w^0$ in the direction $h\in C^{\kappa}([0,T];\mathbb R^d)$, $\kappa\in ]0,1[$ has to satisfy the condition $\alpha +\kappa > 1$ to ensure the existence of the geometric $1/\alpha$-rough path over $w^0 + \varepsilon h$ ($\varepsilon > 0$) provided at \cite{FV10}, Theorem 9.34 when $d > 1$. That condition can be dropped when $d = 1$, because the canonical geometric $1/\alpha$-rough path over $w^0 + \varepsilon h$ is
\begin{equation}\label{natural_geometric_rp}
t\in [0,T]\longmapsto
\left(1,w_{t}^{0} +\varepsilon h_t,\dots,
\frac{(w_{t}^{0} +\varepsilon h_t)^{[1/\alpha]}}{[1/\alpha]!}\right).
\end{equation}
Therefore, since the map $A\mapsto [A(1 - A)]^{\beta}$ is $C^{\infty}$ on $[m_0,M_0]$, $\pi(0,.)$ is continuously differentiable from $B_E((x_{0}^{0},w^0);\eta)$ into $C^0([0,T];\mathbb R)$.
\\
\\
In conclusion, since $(x_{0}^{0},w^0)$ has been arbitrarily chosen, $\pi(0,.)$ is continuously differentiable from $\mathbb R_{+}^{*}\times C^{\alpha}([0,T];\mathbb R)$ into $C^0([0,T];\mathbb R)$.
\end{proof}
\noindent
In the sequel, the solution of equation (\ref{auxiliary_equation_y}) with initial condition $y_0 := F(x_0)$ for $x_0\in ]0,1[$ is denoted by $y(y_0)$ or $y(y_0,w)$.
\\
\\
Let us conclude with the three following corollaries of Proposition \ref{continuous_differentiability}, using the particular form of the vector field of equation (\ref{auxiliary_equation}) :
%


%
\begin{corollary}\label{monotonicity_initial_condition}
Under assumptions \ref{assumption_beta}, \ref{signal_regularity_R_+} and \ref{auxiliary_assumption_cst}, the map $\pi(0,.;w)_t$ is strictly increasing on $]0,1[$ for every $t\in\mathbb R_+$.
\end{corollary}
%


%
\begin{proof}
By Proposition \ref{continuous_differentiability}, for every $t\in\mathbb R_{+}^{*}$ and every $y_0\in ]0,F(1)[$,
\begin{displaymath}
\partial_{y_0}
y_t(y_0) =
\int_{0}^{t}
\exp\left[
\int_{s}^{t}
(G\circ F^{-1})'\left[y_u(y_0)\right]du\right]ds > 0.
\end{displaymath}
Then, $y_0\in ]0,F(1)[\mapsto y_t(y_0)$ is strictly increasing on $]0,F(1)[$ for every $t\in\mathbb R_+$.
\\
\\
Since $F$ and $F^{-1}$ are respectively strictly increasing on $]0,1[$ and $]0,F(1)[$, the map $\pi(0,.;w)_t = F^{-1}[y_t[F(.)]]$ is strictly increasing on $]0,1[$ for every $t\in\mathbb R_+$.
\end{proof}
%


%
\begin{corollary}\label{extension_by_continuity_y}
Under assumptions \ref{assumption_beta}, \ref{signal_regularity_R_+} and \ref{auxiliary_assumption_cst}, there exists two continuous functions $y(0)$ and $y[F(1)]$ (resp. $x(0)$ and $x(1)$) from $\mathbb R_+$ into $[0,F(1)]$ (resp. $[0,1]$) such that :
\begin{eqnarray*}
 & & \lim_{y_0\rightarrow 0}
 \left\|
 y(y_0) - y(0)
 \right\|_{\infty;T} = 0
 \textrm{ and }
 \lim_{y_0\rightarrow F(1)}
 \left\|
 y(y_0) - y[F(1)]
 \right\|_{\infty;T} = 0\\
 \textrm{(resp.} & &
 \lim_{x_0\rightarrow 0}
 \left\|
 \pi(0,x_0;w) - x(0)
 \right\|_{\infty;T} = 0
 \textrm{ and }
 \lim_{x_0\rightarrow 1}
 \left\|
 \pi(0,x_0;w) - x(1)
 \right\|_{\infty;T} = 0\textrm{)}
\end{eqnarray*}
for each $T > 0$. Moreover, $y_t(0)$ and $y_t[F(1)]$ (resp. $x(0)$ and $x(1)$) belong to $]0,F(1)[$ (resp. $]0,1[$) for every $t > 0$.
\end{corollary}
%


%
\begin{proof}
Only the case $y_0\rightarrow 0$ (resp. $x_0\rightarrow 0$) is detailed. The case $y_0\rightarrow F(1)$ (resp. $x_0\rightarrow 1$) is obtained similarly.
\\
\\
On the one hand, as shown at Proposition \ref{Ito_map_continuity} ; for every $y_{0}^{1},y_{0}^{2}\in ]0,F(1)[$,
\begin{displaymath}
\|y(y_{0}^{1}) - y(y_{0}^{2})\|_{\infty}
\leqslant
|y_{0}^{1} - y_{0}^{2}|.
\end{displaymath}
So, $y_0\in ]0,F(1)[\mapsto y(y_0)$ is uniformly continuous from
\begin{displaymath}
(]0,F(1)[,|.|)
\textrm{ into }
(C^0(\mathbb R_+ ; [0,F(1)]),\|.\|_{\infty}),
\end{displaymath}
and since $C^0(\mathbb R_+ ; [0,F(1)])$ equipped with $\|.\|_{\infty}$ is a Banach space, $y_0\in ]0,F(1)[\mapsto y(y_0)$ has a unique continuous extension to $[0,F(1)]$.
\\
\\
On the other hand, for $y_0\in ]0,F(1)[$ and $t > s\geqslant 0$ arbitrarily chosen,
\begin{eqnarray*}
 y_t(y_0) - y_s(y_0)
 -\gamma\theta^{\beta}(w_t - w_s) & = &
 \int_{s}^{t}
 (G\circ F^{-1})\left[y_u(y_0)\right]du\\
 & \geqslant &
 (t - s)(G\circ F^{-1})
 \left[\sup_{u\in [s,t]}
 y_u(y_0)\right],
\end{eqnarray*}
because $G\circ F^{-1}$ is decreasing on $]0,F(1)[$ (cf. Proposition \ref{properties_F}.(6)). Since
\begin{displaymath}
\lim_{y_0\rightarrow 0}
\sup_{u\in [s,t]}
\left|y_u(y_0) - y_u(0)\right| = 0
\end{displaymath}
by construction, if $y_u(0) = 0$ for every $u\in [s,t]$ :
\begin{eqnarray*}
 \lim_{y_0\rightarrow 0}
 \int_{s}^{t}
 (G\circ F^{-1})\left[y_u(y_0)\right]du & = &
 (t - s)
 \lim_{y_0\rightarrow 0}
 (G\circ F^{-1})
 \left[\sup_{u\in [s,t]}y_u(y_0)\right]\\
 & = & \infty
\end{eqnarray*}
and
\begin{displaymath}
\lim_{y_0\rightarrow 0}
y_t(y_0) - y_s(y_0)
-\gamma\theta^{\beta}(w_t - w_s) =
-\gamma\theta^{\beta}(w_s - w_t) < \infty.
\end{displaymath}
Therefore, there exists $u\in [s,t]$ such that $y_u(0) > 0$.
\\
\\
Similarly, since
\begin{eqnarray*}
 y_t(y_0) - y_s(y_0)
 -\gamma\theta^{\beta}(w_t - w_s) & = &
 \int_{s}^{t}
 (G\circ F^{-1})\left[y_u(y_0)\right]du\\
 & \leqslant &
 (t - s)\times\\
 & &
 (G\circ F^{-1})
 \left[F(1) - \sup_{u\in [s,t]}\left[
 F(1) - y_u(y_0)\right]\right],
\end{eqnarray*}
there exists $u\in [s,t]$ such that $y_u(0) < F(1)$.
\\
\\
In particular, there exists a $\mathbb R_+$-valued sequence $(t_{0}^{n},n\in\mathbb N)$ such that $t_{0}^{n}\downarrow 0$ when $n\rightarrow\infty$, and
\begin{displaymath}
y_{t_{0}^{n}}(0)\in ]0,F(1)[
\textrm{ $;$ }
\forall n\in\mathbb N.
\end{displaymath}
Let $n\in\mathbb N$ be arbitrarily chosen. Since $y(0)$ is continuous on $\mathbb R_+$ by construction, $y_t(0)\in ]0,F(1)[$ for every $t\in [t_{0}^{n};\tau_{0,F(1)}(t_{0}^{n})[$ where,
\begin{displaymath}
\tau_{0,F(1)}(t_{0}^{n}) :=
\inf\left\{
t > t_{0}^{n} :
y_t(0) = 0
\textrm{ or }
y_t(0) = F(1)
\right\}.
\end{displaymath}
For $\varepsilon\in ]0,F(1)[$ arbitrarily chosen, by Corollary \ref{monotonicity_initial_condition} together with the continuity of $y(y_0)$ on $\mathbb R_+$ for every $y_0\in [0,\varepsilon]$ ; for every $t\in [t_{0}^{n};\tau_{0,F(1)}(t_{0}^{n})[$, there exists $t_{\min}^{n},t_{\max}^{n}\in [t_{0}^{n},t]$ such that for every $y_0\in [0,\varepsilon]$ and every $s\in [t_{0}^{n},t]$,
\begin{displaymath}
0 < y_{t_{\min}^{n}}(0)
\leqslant y_s(y_0)
\leqslant y_{t_{\max}^{n}}(\varepsilon) < F(1)
\end{displaymath}
and, by Proposition \ref{properties_F}.(6),
\begin{displaymath}
(G\circ F^{-1})[y_{t_{\max}^{n}}(\varepsilon)]
\leqslant
(G\circ F^{-1})[y_s(y_0)]
\leqslant
(G\circ F^{-1})[y_{t_{\min}^{n}}(0)].
\end{displaymath}
Then, by Lebesgue's theorem :
\begin{eqnarray}
 y_{t + t_{0}^{n}}(0) & = &
 y_{t_{0}^{n}}(0) +
 \lim_{y_0\rightarrow 0}
 \int_{t_{0}^{n}}^{t_{0}^{n} + t}
 (G\circ F^{-1})\left[y_s(y_0)\right]ds +
 \gamma\theta^{\beta}
 (w_{t_{0}^{n} + t} - w_{t_{0}^{n}})
 \nonumber\\
 \label{translated_equation_y}
 & = &
 y_{t_{0}^{n}}(0) +
 \int_{0}^{t}
 (G\circ F^{-1})\left[y_{s + t_{0}^{n}}(0)\right]ds +
 \gamma\theta^{\beta}
 (\theta_{t_{0}^{n}}w)_t
\end{eqnarray}
for every $t\in [0;\tau_{0,F(1)}(t_{0}^{n}) - t_{0}^{n}[$ where, $\theta_{t_{0}^{n}}w = w_{t_{0}^{n} + .} - w_{t_{0}^{n}}$. By (\ref{translated_equation_y}) and Theorem \ref{existence_uniqueness_J} :
\begin{eqnarray*}
 \tau_{0,F(1)}(t_{0}^{n}) & = &
 \inf\{
 t > 0 :
 y_t[
 y_{t_{0}^{n}}(0),\theta_{t_{0}^{n}}w] = 0
 \textrm{ or }
 y_t[
 y_{t_{0}^{n}}(0),\theta_{t_{0}^{n}}w] = F(1)\}\\
 & = & \infty.
\end{eqnarray*}
Therefore, $y(0)$ is a $]0,F(1)[$-valued function on $[t_{0}^{n},\infty[$ for every $n\in\mathbb N$. Since $t_{0}^{n}\downarrow 0$ when $n\rightarrow\infty$, $y(0)$ is a $]0,F(1)[$-valued function on $\mathbb R_{+}^{*}$.
\\
\\
By putting $x(0) := F^{-1}[y_.(0)]$, since $\pi(0,x_0;w) = F^{-1}[y_.(y_0)]$ and $F^{-1}$ is continuously differentiable from $[0,F(1)]$ into $[0,1]$, that achieves the proof.
\end{proof}
\noindent
In the sequel, for every $x_0\in [0,1]$,
\begin{displaymath}
x_t(x_0) :=
\left\{
\begin{array}{rcl}
 \lim_{\varepsilon\rightarrow 0}
 \pi\left(0,\varepsilon;w\right)_t
 & \textrm{if} &
 x_0 = 0\\
 \pi\left(0,x_0;w\right)_t
 & \textrm{if} &
 x_0\in ]0,1[\\
 \lim_{\varepsilon\rightarrow 0}
 \pi\left(0,1 -\varepsilon;w\right)_t
 & \textrm{if} &
 x_0 = 1
\end{array}
\right.
\textrm{$;$ }
\forall t\in\mathbb R_+
\end{displaymath}
and for every $y_0\in [0,F(1)]$,
\begin{displaymath}
y_t(y_0) :=
\left\{
\begin{array}{rcl}
 \lim_{\varepsilon\rightarrow 0}
 y_t(\varepsilon)
 & \textrm{if} &
 y_0 = 0\\
 y_t(y_0)
 & \textrm{if} &
 y_0\in ]0,F(1)[\\
 \lim_{\varepsilon\rightarrow 0}
 y_t[F(1) -\varepsilon]
 & \textrm{if} &
 y_0 = F(1)
\end{array}
\right.
\textrm{$;$ }
\forall t\in\mathbb R_+.
\end{displaymath}
%


%
\begin{corollary}\label{Ito_map_Lipschitz}
Under assumptions \ref{assumption_beta}, \ref{signal_regularity_R_+} and \ref{auxiliary_assumption_cst}, there exists two constants $C > 0$ and $l > 0$, only depending on $F$ and $G$, such that :
\begin{displaymath}
|y_t(y_{0}^{1}) - y_t(y_{0}^{2})|
\leqslant
|y_{0}^{1} - y_{0}^{2}|e^{-lt}
\textrm{ $;$ }
\forall t\in\mathbb R_+
\end{displaymath}
for every $y_{0}^{1},y_{0}^{2}\in [0,F(1)]$, and
\begin{displaymath}
|x_t(x_{0}^{1}) -
x_t(x_{0}^{2})|
\leqslant
C|F(x_{0}^{1}) -
F(x_{0}^{2})|e^{-lt}
\textrm{ $;$ }
\forall t\in\mathbb R_+
\end{displaymath}
for every $x_{0}^{1},x_{0}^{2}\in [0,1]$.
\end{corollary}
%


%
\begin{proof}
For $i = 1,2$, consider the solution $x^i$ of equation (\ref{auxiliary_equation}) with initial condition $x_{0}^{i}\in ]0,1[$, and $y^i := F(x_{.}^{i})$.
\\
Moreover, assume that $x_{0}^{1}\not= x_{0}^{2}$. By Corollary \ref{monotonicity_initial_condition}, $y_{t}^{1}\not= y_{t}^{2}$ for every $t\in\mathbb R_+$.
\\
\\
For every $t\in\mathbb R_+$,
\begin{displaymath}
\frac{d}{dt}(y_{t}^{1} - y_{t}^{2}) =
(G\circ F^{-1})(y_{t}^{1}) - (G\circ F^{-1})(y_{t}^{2}).
\end{displaymath}
Then,
\begin{eqnarray*}
 \frac{d}{dt}(y_{t}^{1} - y_{t}^{2})^2 & = &
 2(y_{t}^{1} - y_{t}^{2})
 \left[
 (G\circ F^{-1})(y_{t}^{1}) - (G\circ F^{-1})(y_{t}^{2})\right]\\
 & = &
 2(y_{t}^{1} - y_{t}^{2})^2
 \frac{(G\circ F^{-1})(y_{t}^{1}) - (G\circ F^{-1})(y_{t}^{2})}{y_{t}^{1} - y_{t}^{2}}.
\end{eqnarray*}
By Proposition \ref{properties_F}.(6), there exists a constant $l > 0$ such that :
\begin{displaymath}
\forall z\in ]0,F(1)[
\textrm{, }
(G\circ F^{-1})'(z)\leqslant -l.
\end{displaymath}
So, by the mean value theorem, there exists $c_t\in ]y_{t}^{1}\wedge y_{t}^{2}, y_{t}^{1}\vee y_{t}^{2}[\subset ]0,F(1)[$ such that :
\begin{displaymath}
\frac{(G\circ F^{-1})(y_{t}^{1}) - (G\circ F^{-1})(y_{t}^{2})}{y_{t}^{1} - y_{t}^{2}} =
(G\circ F^{-1})'(c_t)
\leqslant -l.
\end{displaymath}
Therefore,
\begin{displaymath}
\frac{d}{dt}\left(y_{t}^{1} - y_{t}^{2}\right)^2
\leqslant
-2l\left(y_{t}^{1} - y_{t}^{2}\right)^2.
\end{displaymath}
By integrating that inequality :
\begin{displaymath}
\left|y_{t}^{1} - y_{t}^{2}\right|
\leqslant
\left|y_{0}^{1} - y_{0}^{2}\right|e^{-lt}.
\end{displaymath}
Since $x^i = F^{-1}(y_{.}^{i})$ and $F^{-1}$ is continuously differentiable from $[0,F(1)]$ \mbox{into $[0,1]$ :}
\begin{displaymath}
\left|x_t\left(x_{0}^{1}\right) -
x_t\left(x_{0}^{2}\right)\right|
\leqslant
C\left|F\left(x_{0}^{1}\right) -
F\left(x_{0}^{2}\right)\right|e^{-lt}
\end{displaymath}
where, $C > 0$ denotes the Lipschitz constant of $F^{-1}$.
\\
\\
That inequality holds true when $x_{0}^{1}$ or $x_{0}^{2}$ goes to $0$ or $1$, because $F$ is continuous on $[0,1]$.
\end{proof}
%


%
\subsection{Approximation scheme}
In order to provide a converging approximation scheme for equation (\ref{auxiliary_equation}), the convergence of the implicit Euler scheme for equation (\ref{auxiliary_equation_y}) is studied first under assumptions \ref{assumption_beta}, \ref{signal_regularity_R_+} and \ref{auxiliary_assumption_cst}.
\\
\\
Consider the recurrence equation
\begin{equation}\label{euler_scheme_y}
\left\{
\begin{array}{rcl}
y_{0}^{n} & = & y_0\in ]0,F(1)[\\
y_{k + 1}^{n} & = &
\displaystyle{y_{k}^{n} +\frac{T}{n}(G\circ F^{-1})(y_{k + 1}^{n}) +
\gamma\theta^{\beta}(w_{t_{k + 1}^{n}} - w_{t_{k}^{n}}})
\end{array}
\right.
\end{equation}
where, for $n\in\mathbb N^*$ and $T > 0$, $t_{k}^{n} := kT/n$ and $k\leqslant n$ while $y_{k + 1}^{n}\in ]0,F(1)[$.
\\
\\
The following proposition shows that the step-$n$ implicit Euler approximation $y^n$ is defined on $\{0,\dots,n\}$ :
%


%
\begin{proposition}\label{existence_euler_scheme_y}
Under assumptions \ref{assumption_beta}, \ref{signal_regularity_R_+} and \ref{auxiliary_assumption_cst}, equation (\ref{euler_scheme_y}) admits a unique solution $(y^n,n\in\mathbb N^*)$. Moreover,
\begin{displaymath}
\forall n\in\mathbb N^*\textrm{$,$ }
\forall k = 0,\dots, n\textrm{$,$ }
y_{k}^{n}\in ]0,F(1)[.
\end{displaymath}
\end{proposition}
%


%
\begin{proof}
Let $\varphi$ be the function defined on $]0,F(1)[\times\mathbb R\times\mathbb R_{+}^{*}$ by :
\begin{displaymath}
\varphi\left(y,A,B\right) :=
y - B(G\circ F^{-1})(y) - A.
\end{displaymath}
On the one hand, for every $A\in\mathbb R$ and $B > 0$, $\varphi(.,A,B)\in C^{\infty}(]0,F(1)[;\mathbb R)$ and by Proposition \ref{properties_F}.(6), for every $y\in ]0,F(1)[$,
\begin{eqnarray*}
 \partial_y
 \varphi\left(y,A,B\right) & = & 1 - B(G\circ F^{-1})'(y)\\
 & = &
 1 - B\frac{G'[F^{-1}(y)]}{F'[F^{-1}(y)]} > 0.
\end{eqnarray*}
Then, $\varphi(.,A,B)$ is increasing on $]0,F(1)[$. Moreover,
\begin{displaymath}
\lim_{y\rightarrow 0^+}
\varphi\left(y,A,B\right) = -\infty
\textrm{ and }
\lim_{y\rightarrow F(1)^-}
\varphi\left(y,A,B\right) = \infty.
\end{displaymath}
Therefore, since $\varphi$ is continuous on $]0,F(1)[\times\mathbb R\times\mathbb R_{+}^{*}$ :
\begin{displaymath}
\forall A\in\mathbb R\textrm{, }
\forall B > 0\textrm{, }
\exists ! y\in ]0,F(1)[ :
\varphi\left(y,A,B\right) = 0.
\end{displaymath}
On the other hand, for every $n\in\mathbb N^*$, equation (\ref{euler_scheme_y}) can be rewritten as follow :
\begin{equation}\label{euler_scheme_y_rewritten}
\varphi\left[y_{k + 1}^{n},
y_{k}^{n} +
\gamma\theta^{\beta}(
w_{t_{k + 1}^{n}} -
w_{t_{k}^{n}}),
\frac{T}{n}\right] = 0
\textrm{ ; }
k\in\{0,\dots,n\}.
\end{equation}
In conclusion, by recurrence, equation (\ref{euler_scheme_y_rewritten}) admits a unique solution $y_{k + 1}^{n}\in ]0,F(1)[$.
\\
\\
Necessarily, $y_{k}^{n}\in ]0,F(1)[$ for $k = 0,\dots, n$. That achieves the proof.
\end{proof}
\noindent
For each $n\in\mathbb N^*$, consider the function $y^n : [0,T]\rightarrow ]0,F(1)[$ such that
\begin{displaymath}
y_{t}^{n} :=
\sum_{k = 0}^{n - 1}
\left[
y_{k}^{n} +
\frac{y_{k + 1}^{n} - y_{k}^{n}}{t_{k + 1}^{n} - t_{k}^{n}}
(t - t_{k}^{n})\right]
\mathbf 1_{[t_{k}^{n},t_{k + 1}^{n}[}(t)
\end{displaymath}
for every $t\in [0,T]$.
\\
\\
With the ideas of A. Lejay \cite{LEJAY10}, Proposition 5, let us prove that $(y^n,n\in\mathbb N^*)$ converges to the solution of equation (\ref{auxiliary_equation_y}) with initial condition $y_0\in ]0,F(1)[$.
%


%
\begin{theorem}\label{euler_convergence}
Under assumptions \ref{assumption_beta}, \ref{signal_regularity_R_+} and \ref{auxiliary_assumption_cst}, $(y^n,n\in\mathbb N^*)$ is uniformly converging with rate $n^{-\alpha}$ to the solution $y$ of equation (\ref{auxiliary_equation_y}), with initial condition $y_0$, up to the time $T$.
\end{theorem}
%


%
\begin{proof}
The proof follows the same pattern as in \cite{LEJAY10}, Proposition 5.
\\
\\
Consider $n\in\mathbb N^*$, $t\in [0,T]$ and $y$ the solution of equation (\ref{auxiliary_equation_y}) with initial condition $y_0\in ]0,F(1)[$. Since $(t_{k}^{n};k = 0,\dots,n)$ is a subdivision of $[0,T]$, there exists an integer $k\in\{0,\dots,n - 1\}$ such that $t\in [t_{k}^{n},t_{k + 1}^{n}[$.
\\
\\
First of all, note that
\begin{equation}\label{euler_majoration_1}
|y_{t}^{n} - y_t|\leqslant
|y_{t}^{n} - y_{k}^{n}| +
|y_{k}^{n} - z_{k}^{n}| + 
|z_{k}^{n} - y_t|
\end{equation}
where, $z_{i}^{n} := y_{t_{i}^{n}}$ for $i = 0,\dots,n$. Since $y$ is the solution of equation (\ref{auxiliary_equation_y}), $z_{k}^{n}$ and $z_{k + 1}^{n}$ satisfy
\begin{displaymath}
z_{k + 1}^{n} =
z_{k}^{n} +
\frac{T}{n}(G\circ F^{-1})(z_{k + 1}^{n}) +
\gamma\theta^{\beta}
(w_{t_{k + 1}^{n}} -
w_{t_{k}^{n}}) +
\varepsilon_{k}^{n}
\end{displaymath}
where,
\begin{displaymath}
\varepsilon_{k}^{n} :=
\int_{t_{k}^{n}}^{t_{k + 1}^{n}}
[(G\circ F^{-1})(y_s) -
(G\circ F^{-1})(y_{t_{k + 1}^{n}})]ds.
\end{displaymath}
In order to conclude, let us show that $|y_{k}^{n} - z_{k}^{n}|$ is bounded by a quantity not depending on $k$ and converging to $0$ when $n$ goes to infinity.
\\
\\
On the one hand, consider
\begin{displaymath}
y_* :=
\min_{t\in [0,T]} y_t > 0
\textrm{ and }
y^* :=
\max_{t\in [0,T]} y_t < F(1).
\end{displaymath}
Since $G\circ F^{-1}$ is $C^{\infty}$ on $]0,F(1)[$, it is $C_T$-Lipschitz continuous on $[y_*,y^*]$ with $C_T > 0$. Then, for $i = 0,\dots,k$,
\begin{eqnarray}
 |\varepsilon_{i}^{n}|
 & \leqslant &
 \int_{t_{i}^{n}}^{t_{i + 1}^{n}}
 |(G\circ F^{-1})(y_s) -
 (G\circ F^{-1})(y_{t_{i + 1}^{n}})|ds
 \nonumber\\
 & \leqslant &
 C_T\|y\|_{\alpha;T}\int_{t_{i}^{n}}^{t_{i + 1}^{n}}
 |t_{i + 1}^{n} - s|^{\alpha}ds
 \nonumber\\
 & \leqslant &
 \label{euler_majoration_2}
 C_T
 \frac{T^{\alpha + 1}}{\alpha + 1}
 \|y\|_{\alpha;T}
 \frac{1}{n^{\alpha + 1}}.
\end{eqnarray}
On the other hand, let $i\in\{0,\dots,k - 1\}$ be arbitrarily chosen.
\\
\\
Assume that $y_{i + 1}^{n}\geqslant z_{i + 1}^{n}$. Then, by Proposition \ref{properties_F}.(6) :
\begin{displaymath}
(G\circ F^{-1})(y_{i + 1}^{n}) -
(G\circ F^{-1})(z_{i + 1}^{n})
\leqslant 0.
\end{displaymath}
Therefore,
\begin{eqnarray*}
 |y_{i + 1}^{n} - z_{i + 1}^{n}| & = &
 y_{i + 1}^{n} - z_{i + 1}^{n}\\
 & = &
 y_{i}^{n} - z_{i}^{n} +
 \frac{T}{n}\left[
 (G\circ F^{-1})(y_{i + 1}^{n}) -
 (G\circ F^{-1})(z_{i + 1}^{n})\right] -
 \varepsilon_{i}^{n}\\
 & \leqslant &
 |y_{i}^{n} - z_{i}^{n}| + 
 |\varepsilon_{i}^{n}|.
\end{eqnarray*}
Similarly, if $z_{i + 1}^{n} > y_{i + 1}^{n}$, then
\begin{eqnarray*}
 |z_{i + 1}^{n} - y_{i + 1}^{n}| & = &
 z_{i + 1}^{n} - y_{i + 1}^{n}\\
 & \leqslant &
 |y_{i}^{n} - z_{i}^{n}| + 
 |\varepsilon_{i}^{n}|.
\end{eqnarray*}
By putting these cases together :
\begin{equation}\label{euler_majoration_3}
\forall i = 0,\dots, k - 1\textrm{, }
|z_{i + 1}^{n} - y_{i + 1}^{n}|
\leqslant
|z_{i}^{n} - y_{i}^{n}| + 
|\varepsilon_{i}^{n}|.
\end{equation}
By applying (\ref{euler_majoration_3}) recursively from $k - 1$ down to $0$ :
\begin{eqnarray}
 |y_{k}^{n} - z_{k}^{n}| & \leqslant &
 |y_0 - z_0| +
 \sum_{i = 0}^{k - 1}
 |\varepsilon_{i}^{n}|
 \nonumber\\
 \label{euler_majoration_4}
 & \leqslant &
 C_T
 \frac{T^{\alpha + 1}}{\alpha + 1}
 \|y\|_{\alpha;T}
 \frac{1}{n^{\alpha}}
 \xrightarrow[n\rightarrow\infty]{} 0
\end{eqnarray}
because $y_0 = z_0$ and by inequality (\ref{euler_majoration_2}).
\\
\\
Moreover, by (\ref{euler_majoration_4}), there exists $N\in\mathbb N^*$ such that for every integer $n > N$,
\begin{displaymath}
|y_{k + 1}^{n} -
z_{k + 1}^{n}|\leqslant
\max_{i = 1,\dots,n}|y_{i}^{n} - z_{i}^{n}|\leqslant
m_y :=
\frac{y_*}{2}
\end{displaymath}
and
\begin{displaymath}
|y_{k + 1}^{n} -
z_{k + 1}^{n}|\leqslant
\max_{i = 1,\dots,n}|y_{i}^{n} - z_{i}^{n}|\leqslant
\hat M_y :=
M_y - y^*
\end{displaymath}
for any $M_y\in ]y^*,F(1)[$.
\\
\\
In particular,
\begin{displaymath}
m_y\leqslant
y_{t_{k + 1}^{n}} - m_y\leqslant
y_{k + 1}^{n}\leqslant
y_{t_{k + 1}^{n}} + \hat M_y\leqslant M_y.
\end{displaymath}
Since $G\circ F^{-1}$ is a decreasing map by Proposition \ref{properties_F}.(6) :
\begin{displaymath}
(G\circ F^{-1})(M_y)\leqslant
(G\circ F^{-1})(y_{k + 1}^{n})\leqslant
(G\circ F^{-1})(m_y).
\end{displaymath}
Then, by putting $M := |(G\circ F^{-1})(m_y)|\vee |(G\circ F^{-1})(M_y)|$ :
\begin{eqnarray*}
 |y_{t}^{n} - y_{k}^{n}| & = &
 |y_{k + 1}^{n} - y_{k}^{n}|\frac{t - t_{k}^{n}}{t_{k + 1}^{n} - t_{k}^{n}}\\
 & \leqslant &
 \left(TM +
 \gamma\theta^{\beta}T^{\alpha}\|w\|_{\alpha;T}\right)\frac{1}{n^{\alpha}}
 \xrightarrow[n\rightarrow\infty]{} 0.
\end{eqnarray*}
In conclusion, by inequality (\ref{euler_majoration_1}) :
\begin{eqnarray}
 \label{euler_majoration_5}
 |y_{t}^{n} - y_t| & \leqslant &
 \left(TM +
 \gamma\theta^{\beta}T^{\alpha}\|w\|_{\alpha;T} +
 \|y\|_{\alpha;T}\right)\frac{1}{n^{\alpha}} +\\
 & &
 C_T
 \frac{T^{\alpha + 1}}{\alpha + 1}
 \|y\|_{\alpha;T}
 \frac{1}{n^{\alpha}}
 \xrightarrow[n\rightarrow\infty]{} 0
 \nonumber.
\end{eqnarray}
That achieves the proof because the right hand side of inequality (\ref{euler_majoration_5}) is not depending on $k$ or $t$.
\end{proof}
\noindent
Finally, for every $n\in\mathbb N^*$ and $t\in [0,T]$, consider $x_{t}^{n} := F^{-1}(y_{t}^{n})$.
%


%
\begin{corollary}\label{x_approximation_convergence}
Under assumptions \ref{assumption_beta}, \ref{signal_regularity_R_+} and \ref{auxiliary_assumption_cst}, $(x^n,n\in\mathbb N^*)$ is uniformly converging with rate $n^{-\alpha}$ to
\begin{displaymath}
x :=
\pi(0,x_0;w)|_{[0,T]} =
\pi(0,x_0;w|_{[0,T]})
\end{displaymath}
with $x_0\in ]0,1[$.
\end{corollary}
%


%
\begin{proof}
For a given initial condition $x_0 > 0$, it has been shown that $x := F^{-1}(y_.)$ is the solution of equation (\ref{auxiliary_equation}) where, $y$ is the solution of equation (\ref{auxiliary_equation_y}) with initial condition $y_0 := F(x_0)$.
\\
\\
By Theorem \ref{euler_convergence} :
\begin{eqnarray*}
 \|x - x^n\|_{\infty;T} & \leqslant &
 C\|y - y^n\|_{\infty;T}\\
 & \leqslant &
 C\left(TM +
 \gamma\theta^{\beta}T^{\alpha}\|w\|_{\alpha;T} +
 \|y\|_{\alpha;T}\right)\frac{1}{n^{\alpha}} +\\
 & &
 CC_T
 \frac{T^{\alpha + 1}}{\alpha + 1}
 \|y\|_{\alpha;T}
 \frac{1}{n^{\alpha}}
 \xrightarrow[n\rightarrow\infty]{} 0
\end{eqnarray*}
where, $C$ is the Lipschitz constant of $F^{-1}$ on $[0,F(1)]$, since it is continuously differentiable on that interval.
\\
\\
Then, $(x^n,n\in\mathbb N^*)$ is uniformly converging to $x$ with rate $n^{-\alpha}$.
\end{proof}
%


%
\section{Probabilistic properties of Jacobi's equation}
\noindent
Consider a stochastic process $W$ defined on $\mathbb R_+$ and satisfying the following assumption :
%


%
\begin{assumption}\label{assumption_gaussian_process}
$W$ is a $1$-dimensional centered Gaussian process with $\alpha$-H\"older continuous paths on the compact intervals of $\mathbb R_+$ ($\alpha\in ]0,1]$) and $W_0 = 0$.
\end{assumption}
\noindent
For instance, the fractional Brownian motion of Hurst parameter $H\in ]0,1[$ satisfies that assumption for $\alpha\in ]0,H[$.
\\
\\
The canonical probability space of $W$ is denoted by $(\Omega,\mathcal A,\mathbb P)$ with $\Omega := C^0(\mathbb R_+;\mathbb R)$. Under assumptions \ref{assumption_beta} and \ref{auxiliary_assumption_cst}, the solution $\pi(0,x_0;W)$ of equation (\ref{stochastic_auxiliary_equation}) with initial condition $x_0\in ]0,1[$ is defined as the following random variable :
\begin{displaymath}
\pi(0,x_0;W) :=
\left\{
\pi\left[0,x_0; W(\omega)\right];
\omega\in\Omega\right\}.
\end{displaymath}
The regularity of $\pi(0,.)$ studied at propositions \ref{Ito_map_continuity} and \ref{continuous_differentiability}, and corollaries \ref{monotonicity_initial_condition}, \ref{extension_by_continuity_y} and \ref{Ito_map_Lipschitz}, allows to show two probabilistic results on $X :=\pi(0,x_0;W)$ under various additional conditions on $W$ : an ergodic theorem in L. Arnold's random dynamical systems framework with some ideas of M.J. Garrido-Atienza et al. \cite{GAKN09} and B. Schmalfuss \cite{SCHMAL98}, and the existence of an explicit density with respect to Lebesgue's measure on $(\mathbb R,\mathcal B(\mathbb R))$ for each $X_t$, $t > 0$ via I. Nourdin and F. Viens \cite{NV09}. All used results and notations on random dynamical systems and Malliavin calculus are stated in Appendix A.1 and Appendix A.2 respectively.
\\
\\
First of all, without other assumptions on $W$, let us show the \textit{probabilistic} convergence of approximation schemes studied on the deterministic side at \mbox{Section 2 :}
%


%
\begin{proposition}\label{probabilistic_x_approximation_convergence}
Under assumptions \ref{assumption_beta}, \ref{auxiliary_assumption_cst} and \ref{assumption_gaussian_process},
\begin{displaymath}
\lim_{n\rightarrow\infty}
\mathbb E\left(\left\|Y^n - Y\right\|_{\infty;T}^{p}\right) = 0
\textrm{ and }
\lim_{n\rightarrow\infty}
\mathbb E\left(\left\|X^n - X\right\|_{\infty;T}^{p}\right) = 0
\end{displaymath}
for every $p\geqslant 1$ and $T > 0$, where $Y^n$ denotes the step-$n$ implicit Euler approximation scheme of $Y := F(X_.)$ on $[0,T]$, and $X^n := F^{-1}(Y_{.}^{n})$ for each $n\in\mathbb N^*$.
\end{proposition}
%


%
\begin{proof}
By Proposition \ref{euler_convergence} and Corollary \ref{x_approximation_convergence} :
\begin{displaymath}
\left\|Y^n - Y\right\|_{\infty;T}
\xrightarrow[n\rightarrow\infty]{\textrm{a.s.}} 0
\textrm{ and }
\left\|X^n - X\right\|_{\infty;T}
\xrightarrow[n\rightarrow\infty]{\textrm{a.s.}} 0.
\end{displaymath}
Moreover, for every $t\in [0,T]$, $n\in\mathbb N^*$ and $\omega\in\Omega$, $Y_{t}^{n}(\omega),Y_t(\omega)\in ]0,F(1)[$ and $X_{t}^{n}(\omega),X_t(\omega)\in ]0,1[$. Therefore, Lebesgue's theorem allows to conclude.
\end{proof}
%


%
\subsection{An ergodic theorem}
This subsection is devoted to an ergodic theorem for $Y := F(X_.)$ and then $X$, for fractional Brownian signals.
\\
\\
Consider a two-sided fractional Brownian motion $B^H$ of Hurst parameter $H\in ]0,1[$, and $(\Omega,\mathcal A,\mathbb P)$ its canonical probability space with $\Omega := C^0(\mathbb R ;\mathbb R)$. Let $\vartheta := (\theta_t,t\in\mathbb R)$ be the family of maps from the measurable space $(\Omega,\mathcal A)$ into itself, called Wiener shift, such that :
\begin{displaymath}
\forall\omega\in\Omega
\textrm{, }
\forall t\in\mathbb R
\textrm{, }
\theta_t\omega :=
\omega_{t + .} -\omega_t.
\end{displaymath}
By B. Maslowski and B. Schmalfuss \cite{MS04}, $(\Omega,\mathcal A,\mathbb P,\vartheta)$ is an ergodic metric DS.
%


%
\begin{theorem}\label{Jacobi_ergodic_theorem}
Under assumptions \ref{assumption_beta} and \ref{auxiliary_assumption_cst}, let $Y$ be the solution of the following stochastic differential equation :
\begin{equation}\label{auxiliary_two_sided}
Y_t =
Y_0 +
\int_{0}^{t}(G\circ F^{-1})(Y_s)ds +
\gamma\theta^{\beta}B_{t}^{H}
\textrm{ $;$ }
t\in\mathbb R_+
\end{equation}
where, $Y_0 : \Omega\rightarrow ]0,F(1)[$ is an (integrable) random variable.
\begin{enumerate}
 \item There exists an (integrable) random variable $\hat Y :\Omega\rightarrow ]0,F(1)[$ such that
 \begin{displaymath}
 \lim_{t\rightarrow\infty}
 |Y_t(\omega) - \hat Y(\theta_t\omega)| = 0
 \end{displaymath}
 for almost every $\omega\in\Omega$.
 \item For any Lipschitz continuous function $f : [0,F(1)]\rightarrow\mathbb R$,
 \begin{displaymath}
 \lim_{T\rightarrow\infty}
 \frac{1}{T}
 \int_{0}^{T}
 f(Y_t)dt =
 \mathbb E[f(\hat Y)]
 \textrm{ $\mathbb P$-a.s.}
 \end{displaymath}
\end{enumerate}
\end{theorem}
%


%
\begin{proof}
In a first step, the existence of a (generalized) random fixed point is established for the continuous random dynamical system naturally defined by equation (\ref{auxiliary_two_sided}) on the metric space $]0,F(1)[$ over the ergodic metric DS $(\Omega,\mathcal A,\mathbb P,\vartheta)$. The second step is devoted to the ergodic theorem for $Y$ stated at point 2.
\\
\\
\textbf{Step 1.} Let $\varphi :\mathbb R_+\times\Omega\times ]0,F(1)[\rightarrow ]0,F(1)[$ be the map defined by :
\begin{displaymath}
\varphi(t,\omega)x :=
x +
\int_{0}^{t}(G\circ F^{-1})[\varphi(s,\omega)x]ds +
\gamma\theta^{\beta}B_{t}^{H}(\omega).
\end{displaymath}
It is a continuous random dynamical system on $]0,F(1)[$ over the metric DS $(\Omega,\mathcal A,\mathbb P,\vartheta)$. Indeed, for every $s,t\in\mathbb R_+$, $\omega\in\Omega$ and $x\in ]0,F(1)[$ ; $\varphi(0,\omega)x = x$ and
\begin{eqnarray*}
 \varphi(s + t,\omega)x & = &
 x +
 \int_{0}^{s + t}(G\circ F^{-1})[\varphi(u,\omega)x]du +
 \gamma\theta^{\beta}B_{s + t}^{H}(\omega)\\
 & = &
 \varphi(s,\omega)x +
 \int_{s}^{s + t}(G\circ F^{-1})[\varphi(u,\omega)x]du +
 \gamma\theta^{\beta}[B_{s + t}^{H}(\omega) - B_{s}^{H}(\omega)]\\
 & = &
 \varphi(s,\omega)x +
 \int_{0}^{t}(G\circ F^{-1})[\varphi(s + u,\omega)x]du +
 \gamma\theta^{\beta}
 B_{t}^{H}(\theta_s\omega).
\end{eqnarray*}
Then, $\varphi(s + t,\omega) = \varphi(t,\theta_s\omega)\circ\varphi(s,\omega)$. In other words, $\varphi$ satisfies the cocycle property. Proposition \ref{Ito_map_continuity} allows to conclude. That RDS has two additional \mbox{properties :}
\begin{itemize}
 \item\textbf{Additional property 1.} By Corollary \ref{extension_by_continuity_y}, for every $t\in\mathbb R_+$ and $\omega\in\Omega$, the limits
 \begin{displaymath}
 \varphi(t,\omega)0 :=
 \lim_{x\rightarrow 0}
 \varphi(t,\omega)x
 \textrm{ and }
 \varphi(t,\omega)F(1) :=
 \lim_{x\rightarrow F(1)}
 \varphi(t,\omega)x
 \end{displaymath}
 exist, and belong to $]0,F(1)[$ if and only if $t > 0$.
 \item\textbf{Additional property 2.} By Corollary \ref{Ito_map_Lipschitz}, there exists $l > 0$ such that for every $t\in\mathbb R_+$, $\omega\in\Omega$ and $x,y\in ]0,F(1)[$,
 \begin{displaymath}
 |\varphi(t,\omega)x -
 \varphi(t,\omega)y|
 \leqslant
 e^{-lt}|x - y|.
 \end{displaymath}
\end{itemize}
Let us now show that there exists a random variable $\hat Y :\Omega\rightarrow ]0,F(1)[$ such that
\begin{displaymath}
\varphi(t,\omega)\hat Y(\omega) =
\hat Y(\theta_t\omega)
\end{displaymath}
for every $t\in\mathbb R_+$ and $\omega\in\Omega$.
\\
\\
On the one hand, by the cocycle property of $\varphi$ together with \textit{additional property 2}, for every $n\in\mathbb N$, $\omega\in\Omega$ and $x\in ]0,F(1)[$,
\begin{eqnarray*}
 |\varphi(n,\theta_{-n}\omega)x -
 \varphi(n + 1,\theta_{-(n + 1)}\omega)x| & = &
 |\varphi(n,\theta_{-n}\omega)x -\\
 & &
 [\varphi(n,\theta_{-n}\omega)\circ
 \varphi(1,\theta_{-(n + 1)}\omega)]x|\\
 & \leqslant &
 e^{-ln}|x - \varphi(1,\theta_{-(n + 1)}\omega)x|\\
 & \leqslant & F(1)e^{-ln}.
\end{eqnarray*}
Then, $\{\varphi(n,\theta_{-n}\omega)x ; n\in\mathbb N\}$ is a Cauchy sequence, and its limit $\hat Y(\omega)$ is not depending on $x$, because for any other $y\in ]0,F(1)[$,
\begin{displaymath}
|\varphi(n,\theta_{-n}\omega)x -
\varphi(n,\theta_{-n}\omega)y|
\leqslant
e^{-ln}|x - y|
\xrightarrow[n\rightarrow\infty]{} 0.
\end{displaymath}
Moreover, for every $t\in\mathbb R_+$,
\begin{eqnarray*}
 |\varphi(t,\theta_{-t}\omega)x - \hat Y(\omega)| & \leqslant &
 |\varphi(t,\theta_{-t}\omega)x -\varphi([t],\theta_{-[t]}\omega)x| +
 |\varphi([t],\theta_{-[t]}\omega)x -\hat Y(\omega)|\\
 & \leqslant &
 |[\varphi([t],\theta_{-[t]}\omega)\circ\varphi(t - [t],\theta_{-t}\omega)]x
 -\varphi([t],\theta_{-[t]}\omega)x| +\\
 & &
 |\varphi([t],\theta_{-[t]}\omega)x -\hat Y(\omega)|\\
 & \leqslant &
 e^{-l[t]}
 |\varphi(t - [t],\theta_{-t}\omega)x - x| +
 |\varphi([t],\theta_{-[t]}\omega)x -\hat Y(\omega)|\\
 & \leqslant &
 F(1)e^{-l[t]} +
 |\varphi([t],\theta_{-[t]}\omega)x -\hat Y(\omega)|
 \xrightarrow[t\rightarrow\infty]{} 0.
\end{eqnarray*}
Therefore,
\begin{equation}\label{hat_Y_definition}
\lim_{t\rightarrow\infty}
|\varphi(t,\theta_{-t}\omega)x -\hat Y(\omega)| = 0.
\end{equation}
On the other hand, by the cocycle property of $\varphi$, for every $t\in\mathbb R_+$, $n\in\mathbb N$, $\omega\in\Omega$ and $x\in ]0,F(1)[$,
\begin{eqnarray}
 \label{cocycle_application}
 [\varphi(t,\omega)\circ
 \varphi(n,\theta_{-n}\omega)]x
 & = &
 \varphi(t + n,\theta_{-n}\omega)x\\
 & = &
 \varphi[t + n,(\theta_{-(t + n)}\circ\theta_t)\omega]x
 \nonumber.
\end{eqnarray}
The continuity of the random dynamical system $\varphi$, \textit{additional property 1} and (\ref{hat_Y_definition}) imply that :
\begin{itemize}
 \item When $n\rightarrow\infty$ in equality (\ref{cocycle_application}) :
 \begin{displaymath}
 \varphi(t,\omega)\hat Y(\omega) =
 \hat Y(\theta_t\omega).
 \end{displaymath}
 \item By replacing $\omega$ by $\theta_{-t}\omega$ in equality (\ref{cocycle_application}) for $t > 0$, when $n\rightarrow\infty$ :
 \begin{displaymath}
 \varphi(t,\theta_{-t}\omega)
 \hat Y(\theta_{-t}\omega) =
 \hat Y(\omega)\in ]0,F(1)[.
 \end{displaymath}
\end{itemize}
Since $(\Omega,\mathcal A,\mathbb P,\vartheta)$ is an ergodic metric DS and $\hat Y$ is a (generalized) random fixed point of the continuous RDS $\varphi$, $(\hat Y\circ\theta_t,t\in\mathbb R_+)$ is a stationary solution of equation (\ref{auxiliary_two_sided}). Therefore, for almost every $\omega\in\Omega$,
\begin{displaymath}
\lim_{t\rightarrow\infty}
|Y_t(\omega) - \hat Y(\theta_t\omega)| = 0
\end{displaymath}
because all solutions of equation (\ref{auxiliary_two_sided}) converge pathwise forward to each other in time by \textit{additional property 2}.
\\
\\
\textbf{Step 2.} Let $f : [0,F(1)]\rightarrow\mathbb R$ be Lipschitz continuous. For every $T > 0$ and $\omega\in\Omega$,
\begin{displaymath}
\frac{1}{T}\int_{0}^{T}
f\left[Y_t(\omega)\right]dt =
A_T(\omega) + B_T(\omega)
\end{displaymath}
where,
\begin{displaymath}
A_T(\omega) :=
\frac{1}{T}
\int_{0}^{T}
f[\hat Y(\theta_t\omega)]dt
\textrm{ and }
B_T(\omega) :=
\frac{1}{T}
\int_{0}^{T}[
f[Y_t(\omega)] - f[\hat Y(\theta_t\omega)]]dt.
\end{displaymath}
On the one hand, since $(\Omega,\mathcal A,\mathbb P,\vartheta)$ is an ergodic metric DS, by the Birkhoff-Chintchin's theorem (Theorem \ref{ergodic_theorem}) :
\begin{displaymath}
\lim_{T\rightarrow\infty}
A_T =
\mathbb E[f(\hat Y)]
\textrm{ $\mathbb P$-a.s.}
\end{displaymath}
On the other hand, since $f$ is Lipschitz continuous on $[0,F(1)]$, the first step of the proof implies that for almost every $\omega\in\Omega$ and $\varepsilon > 0$ arbitrarily chosen, there exists $T_0 > 0$ such that :
\begin{displaymath}
\forall t > T_0\textrm{, }
|f[Y_t(\omega)] - f[\hat Y(\theta_t\omega)]|
\leqslant
\frac{\varepsilon}{2}.
\end{displaymath}
Then, for every $T > T_0$,
\begin{eqnarray*}
 |B_T(\varepsilon)| & \leqslant &
 \frac{1}{T}\int_{0}^{T_0}
 |f[Y_t(\omega)] - f[\hat Y(\theta_t\omega)]|dt +
 \frac{1}{T}\int_{T_0}^{T}
 |f[Y_t(\omega)] - f[\hat Y(\theta_t\omega)]|dt\\
 & \leqslant &
 \frac{1}{T}\int_{0}^{T_0}
 |f[Y_t(\omega)] - f[\hat Y(\theta_t\omega)]|dt +
 \frac{\varepsilon}{2}.
\end{eqnarray*}
Moreover, there exists $T_1\geqslant T_0$ such that :
\begin{displaymath}
\forall T > T_1
\textrm{, }
\frac{1}{T}\int_{0}^{T_0}
|f[Y_t(\omega)] - f[\hat Y(\theta_t\omega)]|dt\leqslant
\frac{\varepsilon}{2}.
\end{displaymath}
Therefore,
\begin{displaymath}
\lim_{T\rightarrow\infty}
B_T = 0
\textrm{ $\mathbb P$-a.s.}
\end{displaymath}
That achieves the proof.
\end{proof}
\noindent
\textbf{Remarks :} With notations of Theorem \ref{Jacobi_ergodic_theorem} :
\begin{enumerate}
 \item Consider $\mu_t$ and $\hat\mu_t$ the respective probability distributions of $Y_t$ and $\hat Y\circ\theta_t$ under $\mathbb P$ for every $t\in\mathbb R_+$. Since $(\hat Y\circ\theta_t,t\in\mathbb R_+)$ is a stationary process by construction, $\hat\mu_t =\hat\mu$ for each $t\in\mathbb R_+$, where $\hat\mu$ denotes the probability distribution of $\hat Y$ under $\mathbb P$.
 \\
 By Kantorovich-Rubinstein's dual representation of the Wasserstein metric $W_1$ (cf. \cite{VILLANI08}, Theorem 5.10), Theorem \ref{Jacobi_ergodic_theorem}.(1) and Lebesgue's theorem :
 \begin{eqnarray*}
  W_1(\mu_t,\hat\mu) & = &
  W_1(\mu_t,\hat\mu_t)\\
  & = &
  \sup\left\{
  \int_{\bar S} f(y)(\mu_t -\hat\mu_t)(dy) ;
  f : \bar S\rightarrow\mathbb R
  \textrm{ Lipschitz with constant $1$}
  \right\}\\
  & \leqslant &
  \mathbb E(|Y_t -\hat Y\circ\theta_t|)
  \xrightarrow[t\rightarrow\infty]{} 0
 \end{eqnarray*}
 where, $S := ]0,F(1)[$.
 In particular,
 \begin{displaymath}
 Y_t
 \xrightarrow[t\rightarrow\infty]{d}\hat Y.
 \end{displaymath}
 \item By Corollary \ref{extension_by_continuity_y}, the map $\hat\varphi :\mathbb R_+\times\Omega\times\bar S\rightarrow\bar S$ defined by
 \begin{displaymath}
 \hat\varphi(t,\omega)x :=
 \left\{
 \begin{array}{rcl}
  \varphi(t,\omega)0 & \textrm{if} & x = 0\\
  \varphi(t,\omega)x & \textrm{if} & x\in ]0,F(1)[\\
  \varphi(t,\omega)F(1) & \textrm{if} & x = F(1)
 \end{array}
 \right.
 \end{displaymath}
 is a continuous RDS on $\bar S$ over the metric DS $(\Omega,\mathcal A,\mathbb P,\vartheta)$. Since $\bar S$ is a compact metric space, by Theorem \ref{invariant_measure_compact}, there exists at least one $\hat\varphi$-invariant probability measure (cf. Definition \ref{invariant_measures_RDS}).
 \\
 \\
 Theorem \ref{Jacobi_ergodic_theorem} allows to get an explicit $\varphi$-invariant probability measure and its factorization with respect to $\mathbb P$ :
 \\
 Consider $\Theta$ the skew product of the metric DS $(\Omega,\mathcal A,\mathbb P,\vartheta)$ and the RDS $\varphi$ on $S$, and the measure $\mu\in\mathcal P_{\mathbb P}(\Omega\times S)$ defined by
 \begin{displaymath}
 \mu(d\omega,dx) :=
 \delta_{\{\hat Y(\omega)\}}(dx)\mathbb P(d\omega).
 \end{displaymath}
 Since $S$ is a Polish space, $(\Omega,\mathcal A,\mathbb P,\vartheta)$ is an ergodic metric DS and $\hat Y$ is a (generalized) random fixed point of the continuous RDS $\varphi$ ; for every continuous and bounded map $f :\Omega\times S\rightarrow\mathbb R$ and every $t\in\mathbb R_+$,
 \begin{eqnarray*}
  (\Theta_t\mu)(f) & = &
  \int_{\Omega\times S}
  f[\Theta_t(\omega,x)]
  \mu(d\omega,dx)\\
  & = &
  \int_{\Omega\times S}
  f[\theta_t\omega,\varphi(t,\omega)x]
  \delta_{\{\hat Y(\omega)\}}(dx)\mathbb P(d\omega)\\
  & = &
  \int_{\Omega}
  f[\theta_t\omega,\varphi(t,\omega)\hat Y(\omega)]\mathbb P(d\omega)\\
  & = &
  \int_{\Omega}
  f[\theta_t\omega,\hat Y(\theta_t\omega)]\mathbb P(d\omega)\\
  & = &
  \int_{\Omega}
  f[\omega,\hat Y(\omega)]\mathbb P(d\omega)\\
  & = &
  \int_{\Omega\times S}
  f(\omega,x)\delta_{\{\hat Y(\omega)\}}(dx)\mathbb P(d\omega)
  = \mu(f).
 \end{eqnarray*}
 Therefore, $\mu$ is a $\varphi$-invariant probability measure.
\end{enumerate}
%


%
\begin{corollary}\label{Jacobi_ergodic_corollary}
Under assumptions \ref{assumption_beta} and \ref{auxiliary_assumption_cst}, consider
\begin{displaymath}
X :=
\pi\left(0,X_0;B^H\right)
\end{displaymath}
where, $X_0 : \Omega\rightarrow ]0,1[$ is an (integrable) random variable.
\begin{enumerate}
 \item There exists an (integrable) random variable $\hat X :\Omega\rightarrow ]0,1[$ such that
 \begin{displaymath}
 \lim_{t\rightarrow\infty}
 |X_t(\omega) - \hat X(\theta_t\omega)| = 0
 \end{displaymath}
 for almost every $\omega\in\Omega$.
 \item For any Lipschitz continuous function $f : [0,1]\rightarrow\mathbb R$,
 \begin{displaymath}
 \lim_{T\rightarrow\infty}
 \frac{1}{T}
 \int_{0}^{T}
 f(X_t)dt =
 \mathbb E[f(\hat X)]
 \textrm{ $\mathbb P$-a.s.}
 \end{displaymath}
\end{enumerate}
\end{corollary}
%


%
\begin{proof}
For every $t\in\mathbb R_+$, $X_t = F^{-1}(Y_t)$ where $Y$ is the solution of equation (\ref{auxiliary_two_sided}) with initial condition $Y_0 := F(X_0)$. Since $F^{-1}$ is continuously differentiable from $[0,F(1)]$ into $[0,1]$, by putting $\hat X := F^{-1}(\hat Y)$, Corollary \ref{Jacobi_ergodic_corollary} is a straightforward application of Theorem \ref{Jacobi_ergodic_theorem}.
\end{proof}
\noindent
This subsection concludes on numerical illustrations of points 1 and 2 of Corollary \ref{Jacobi_ergodic_corollary}. \textit{WaveLab802} (\textit{Scilab} package) is used to get wavelet-based simulations of the fractional Brownian motion (cf. \cite{DIEKER04}, Section 2.2.5).
\begin{enumerate}
 \item The converging approximations provided at Theorem \ref{euler_convergence} and Corollary \ref{x_approximation_convergence} are computed with the following values of the parameters :
 \begin{center}
 \begin{tabular}{|l l|}
 \hline
 Parameters & Values\\
 \hline\hline
 $T$ & $10$\\
 $H$ & $0.6$\\
 $\beta$, $\mu$ & $0.5$\\
 $\theta$, $\gamma$ & $1$\\
 \hline
 $n$ & $250$\\
 \hline
 \end{tabular}
 \end{center}
 For one sample path $B^H(\omega)$ ($\omega\in\Omega$) of the fractional Brownian motion $B^H$ of Hurst parameter $H$, the solution $\pi[0,x_0;B^H(\omega)]$ is approximated on $[0,T]$ for $x_0 = 0.01,0.28,0.89$ :
 \begin{figure}[H]
 \centering
 \includegraphics[scale = 0.45]{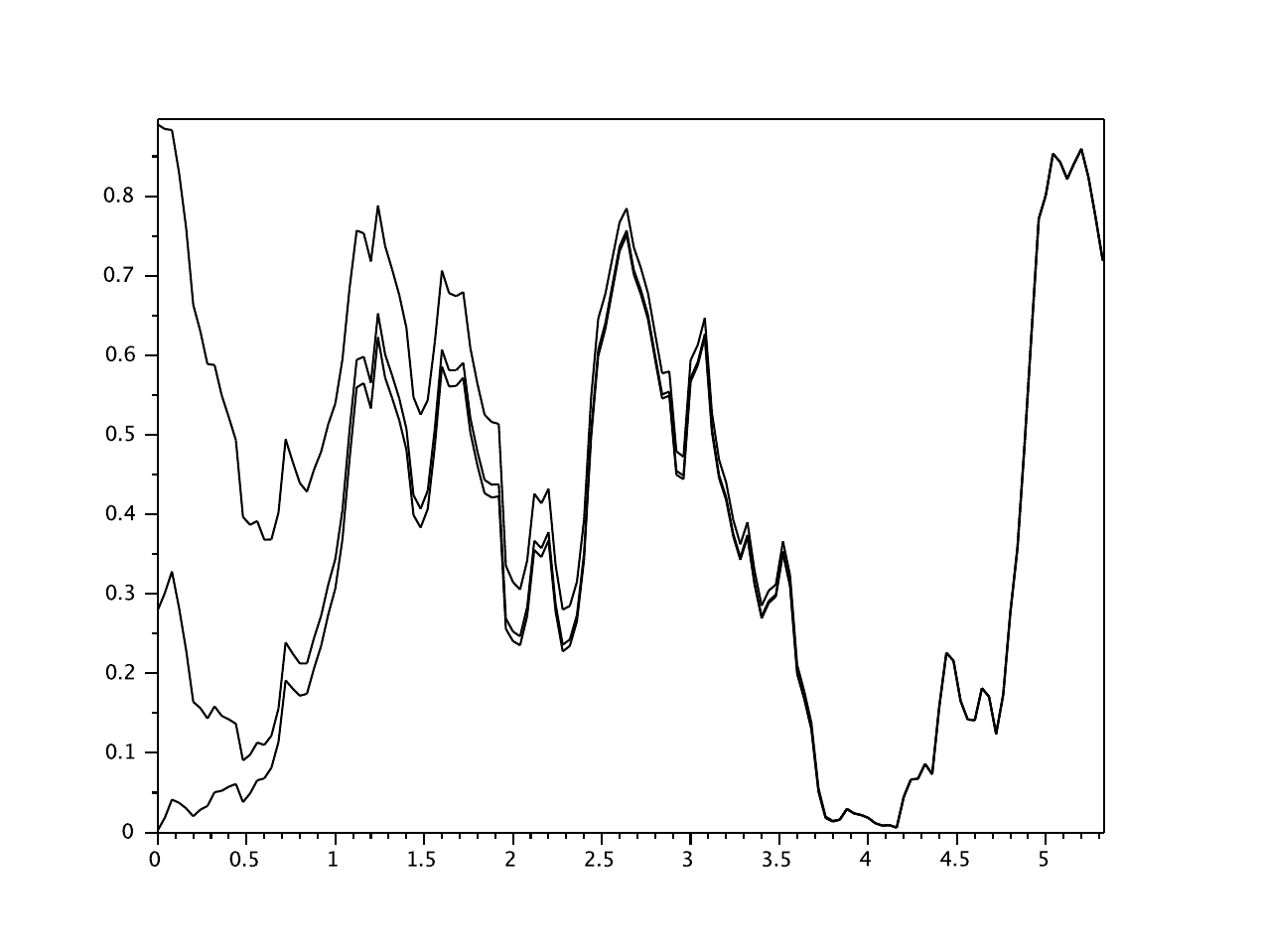}
 \caption{$t\in [0,T]\mapsto\pi[0,x_0;B^H(\omega)]_t$ for $x_0 = 0.01,0.28,0.89$}
 \end{figure}
 \item The converging approximations provided at Theorem \ref{euler_convergence} and Corollary \ref{x_approximation_convergence} are computed with the following values of the parameters :
 \begin{center}
 \begin{tabular}{|l l|}
 \hline
 Parameters & Values\\
 \hline\hline
 $T$ & $120$\\
 $H$ & $0.6$\\
 $\beta$, $\mu$ & $0.5$\\
 $\theta$, $\gamma$ & $1$\\
 $x_0$ & $F^{-1}(1.5)$\\
 \hline
 $n$ & $850$\\
 \hline
 \end{tabular}
 \end{center}
 Consider
 \begin{displaymath}
 S_t :=
 \frac{1}{t}
 \int_{0}^{t}\pi\left(0,x_0;B^H\right)_sds
 \textrm{ $;$ }\forall t\in [0,T].
 \end{displaymath}
 Approximations of four sample paths of the process $S$ on $[0,T]$ :
 \begin{figure}[H]
 \centering
 \includegraphics[scale = 0.45]{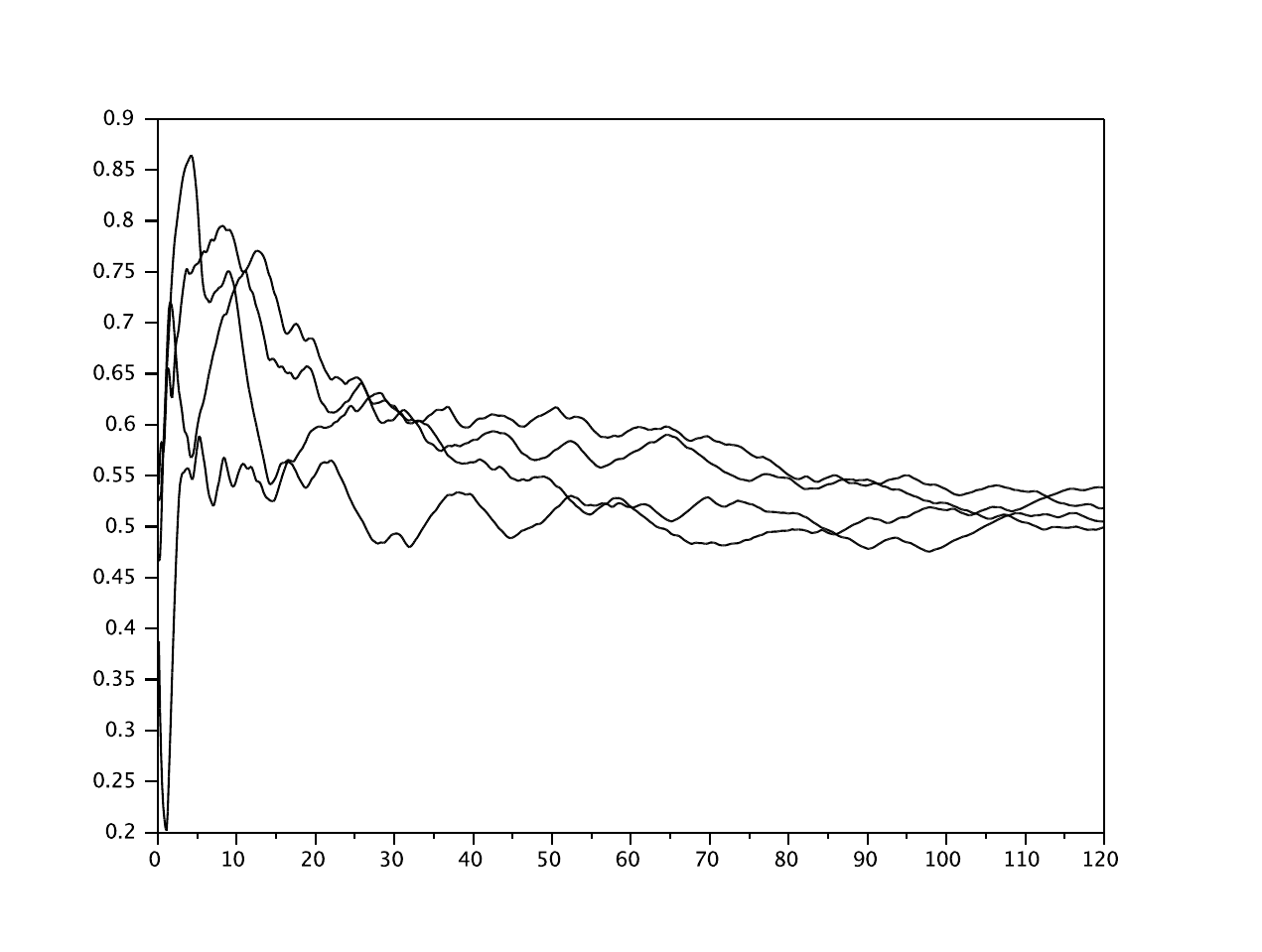}
 \caption{Sample paths of the process $S$ on $[0,T]$}
 \end{figure}
\end{enumerate}
\noindent
\textbf{Remark.} Note that for $\beta = 0.5$ ; $F$, $F^{-1}$ and $b$ can be computed faster because they have explicit expressions :
\begin{eqnarray*}
 F(x) & = & \frac{\pi}{2} +\arcsin(2x - 1)\textrm{ ; $\forall x\in [0,1]$,}\\
 F^{-1}(y) & = & \frac{1}{2}\left[\sin\left(y -\frac{\pi}{2}\right) + 1\right]\textrm{ ; $\forall y\in [0,\pi]$ and}\\
 (G\circ F^{-1})(y) & = & \frac{2\mu - 1 -\sin\left(y -\pi/2\right)}{\cos\left(y -\pi/2\right)}\textrm{ ; $\forall y\in [0,\pi]$.}
\end{eqnarray*}
%


%
\subsection{Explicit density with respect to Lebesgue's measure}
Consider $T > 0$, $t\in ]0,T]$ and, under assumptions \ref{assumption_beta} and \ref{auxiliary_assumption_cst}, $Y_t := F(X_t)$ satisfying
\begin{equation}\label{auxiliary_equation_y_stochastic}
Y_t =
y_0 +
\int_{0}^{t}
G\left(X_s\right)ds +
\gamma\theta^{\beta}W_t
\end{equation}
where, $y_0 := F(x_0)$ and $W$ is a stochastic process satisfying the following assumption :
%


%
\begin{assumption}\label{Gaussian_assumption}
$W$ is a $1$-dimensional centered Gaussian process defined on $[0,T]$, with $\alpha$-H\"older continuous paths and $W_0 = 0$, such that :
\begin{enumerate}
 \item The covariance function $R$ of $W$ satisfies $R(t,t) > 0$ for every $t\in ]0,T]$.
 \item $\langle\varphi_1,\psi_1\rangle_{\mathcal H}\geqslant\langle\varphi_2,\psi_2\rangle_{\mathcal H}$ for every $\varphi_1,\varphi_2,\psi_1,\psi_2\in\mathcal H$ such that
 \begin{displaymath}
 \varphi_1(t)\geqslant\varphi_2(t)\geqslant 0\textrm{ and }
 \psi_1(t)\geqslant\psi_2(t)\geqslant 0\textrm{ $;$ }
 \forall t\in [0,T].
 \end{displaymath}
\end{enumerate}
\end{assumption}
\noindent
\textbf{Example.} A fractional Brownian motion $B^H$ of Hurst parameter $H\in ]0,1[$ satisfies Assumption \ref{Gaussian_assumption} for every $\alpha\in ]0,H[$ (cf. D. Nualart \cite{NUALART06}, Section 5.1.3).
\\
Precisely, $B^H$ satisfies Assumption \ref{Gaussian_assumption} because of the expression of the scalar product $\langle .,.\rangle_{\mathcal H}$ on the reproducing kernel Hilbert space $\mathcal H$ of $B^H$ recalled in Appendix A.2.
%


%
\begin{lemma}\label{integrability_Malliavin_matrix}
Under assumptions \ref{assumption_beta}, \ref{auxiliary_assumption_cst} and \ref{Gaussian_assumption}, the random variable $Y_t$ belongs to $\mathbb D^{1,2}$ and
\begin{displaymath}
\mathbf D_.Y_t =
\gamma\theta^{\beta}
\mathbf 1_{[0,t]}(.)
\exp\left[\int_{.}^{t}
\frac{G'\left(X_u\right)}{F'\left(X_u\right)}du\right].
\end{displaymath}
\end{lemma}
%


%
\begin{proof}
Since $W(\omega + h) = W(\omega) + h$ for every $(\omega,h)\in\Omega\times\mathcal H^1$ and $\mathcal H^1\hookrightarrow C^{\alpha}([0,T];\mathbb R)$ ($\alpha\in ]0,H[$) ; by Proposition \ref{continuous_differentiability}, $Y_s$ is continuously $\mathcal H^1$-differentiable for every $s\in ]0,t]$. Then, by Proposition \ref{link_H1_differentiability_Malliavin_derivability}, $Y_t\in\mathbb D_{\textrm{loc}}^{1,2}$. Moreover, for every $h\in\mathcal H^1$,
\begin{displaymath}
D_hY_t =
\gamma\theta^{\beta}h_t +
\int_{0}^{t}
\frac{G'\left(X_s\right)}{F'\left(X_s\right)}D_hY_sds.
\end{displaymath}
(cf. the remark following Proposition \ref{link_H1_differentiability_Malliavin_derivability} for the notation).
\\
\\
Let $(h^n,n\in\mathbb N)$ be an orthonormal basis of the reproducing kernel Hilbert \mbox{space $\mathcal H$ :}
\begin{eqnarray*}
 \mathbf D_.Y_t & = &
 \sum_{n\in\mathbb N}
 \langle\mathbf DY_t,h^n\rangle_{\mathcal H}h^n =
 \sum_{n\in\mathbb N}
 \left[D_{I(h^n)}Y_t\right]h^n\\
 & = &
 \sum_{n\in\mathbb N}
 \left[
 \gamma\theta^{\beta}I_t(h^n) +
 \int_{0}^{t}\frac{G'(X_s)}{F'(X_s)}D_{I(h^n)}Y_sds\right]h^n\\
 & = &
 \gamma\theta^{\beta}
 \sum_{n\in\mathbb N}\left[D_{I(h^n)}W_t\right]h^n +
 \int_{0}^{t}\frac{G'(X_s)}{F'(X_s)}
 \left[\sum_{n\in\mathbb N}\langle\mathbf DY_s,h^n\rangle_{\mathcal H}h^n\right]ds\\
 & = &
 \gamma\theta^{\beta}\mathbf D_.W_t +
 \int_{0}^{t}\frac{G'(X_s)}{F'(X_s)}\mathbf D_.Y_s ds.
\end{eqnarray*}
Since $\mathbf D_vY_.$ is the solution of a linear differential equation for every $v\in [0,T]$ :
\begin{displaymath}
\mathbf D_.Y_t =
\gamma\theta^{\beta}
\mathbf 1_{[0,t]}(.)
\exp\left[\int_{.}^{t}\frac{G'(X_u)}{F'(X_u)}du\right].
\end{displaymath}
So, since $G'(y)/F'(y) < 0$ for every $y\in ]0,1[$ (cf. Proposition \ref{properties_F}.(6)) :
\begin{equation}\label{estimate_Malliavin_1}
\gamma\theta^{\beta}\mathbf 1_{[0,t]}(s)
\exp\left[\int_{0}^{T}\frac{G'(X_u)}{F'(X_u)}du\right]
\leqslant
\mathbf D_sY_t
\leqslant
\gamma\theta^{\beta}\mathbf 1_{[0,t]}(s)
\end{equation}
for every $s\in [0,T]$. Put $\gamma_t :=\|\mathbf DY_t\|_{\mathcal H}^{2}$. By Assumption \ref{Gaussian_assumption} together with inequality (\ref{estimate_Malliavin_1}) :
\begin{equation}\label{estimate_Malliavin_2}
0 < (\gamma\theta^{\beta})^2R(t,t)
\exp\left[2\int_{0}^{T}\frac{G'(X_u)}{F'(X_u)}du\right]
\leqslant\gamma_t
\leqslant
(\gamma\theta^{\beta})^2R(t,t).
\end{equation}
By inequality (\ref{estimate_Malliavin_2}), $\gamma_t\in L^p(\Omega,\mathbb P)$ for every $p > 0$. So, $Y_t\in\mathbb D^{1,2}$ by Proposition \ref{link_H1_differentiability_Malliavin_derivability}.
\end{proof}
\noindent
Even if the pathwise properties of $Y$ are sufficient to show the existence of a density for $Y_t$ via Bouleau-Hirsch's criterion (cf. \cite{NUALART06}, Theorem 2.1.2), several probabilistic integrability properties are required in order to provide an expression of the density. If the random variable is derivable in Malliavin's sense and the inverse of the Malliavin matrix belongs to $L^p(\Omega;\mathbb P)$ for each $p\geqslant 1$ ; D. Nualart \cite{NUALART06}, Proposition 2.1.1 provides an explicit density. However, even if $Y_t\in\mathbb D^{1,2}$ by the previous lemma, it seems difficult to show that $1/\gamma_t$ belongs to $L^p(\Omega;\mathbb P)$ too ($\gamma_t :=\|\mathbf DY_t\|_{\mathcal H}^{2}$). I. Nourdin and F. Viens \cite{NV09}, Theorem 3.1 provides an expression of the density, in which the Ornstein-Uhlenbeck operator $L$ (cf. Definition \ref{OU_operator}) involves, but not the inverse of the Malliavin matrix and the divergence operator. The following proposition shows that $Y_t$ satisfies assumptions of \cite{NV09}, Theorem 3.1 :
%


%
\begin{proposition}\label{expression_of_the_density}
Under assumptions \ref{assumption_beta}, \ref{auxiliary_assumption_cst} and \ref{Gaussian_assumption}, the following function $f_t$ is a density of $Y_t$ with respect to Lebesgue's measure \mbox{on $(\mathbb R,\mathcal B(\mathbb R))$ :}
\begin{displaymath}
f_t(y) =
\frac{\mathbb E(|\hat Y_t|)}{2g_{Y_t}(y)}
\exp\left[-\int_{\mathbb E(Y_t)}^{y}
\frac{z -\mathbb E(Y_t)}{g_{Y_t}(z)}dz\right]
\end{displaymath}
where, $\hat Y_t := Y_t - \mathbb E(Y_t)$ and $g_{Y_t}(y) := \mathbb E(\langle\mathbf DY_t,-\mathbf DL^{-1}Y_t\rangle_{\mathcal H}|Y_t = y)$ for every $y\in ]0,F(1)[$.
\end{proposition}
%


%
\begin{proof}
Let $s\in [0,T]$ be arbitrarily chosen. As shown in the proof of \cite{NV09}, Proposition 3.7 :
\begin{displaymath}
-\mathbf D_sL^{-1}Y_t =
\int_{0}^{\infty}
e^{-u}T_u(\mathbf D_sY_t)du.
\end{displaymath}
So, by inequality (\ref{estimate_Malliavin_1}) and the remark following Definition \ref{OU_operator} :
\begin{displaymath}
-\mathbf D_sL^{-1}Y_t
\geqslant
\gamma\theta^{\beta}\mathbf 1_{[0,t]}(s)
\int_{0}^{\infty}e^{-u}T_u\left[\exp\left[\int_{0}^{T}\frac{G'(X_v)}{F'(X_v)}dv\right]\right]du.
\end{displaymath}
Then, by Assumption \ref{Gaussian_assumption} together with inequality (\ref{estimate_Malliavin_1}) :
\begin{eqnarray*}
 \langle\mathbf DY_t,-\mathbf DL^{-1}Y_t\rangle_{\mathcal H}
 & \geqslant &
 (\gamma\theta^{\beta})^2R(t,t)
 \exp\left[\int_{0}^{T}\frac{G'(X_v)}{F'(X_v)}dv\right]\times\\
 & &
 \int_{0}^{\infty}e^{-u}T_u\left[\exp\left[\int_{0}^{T}\frac{G'(X_v)}{F'(X_v)}dv\right]\right]du > 0.
\end{eqnarray*}
So,
\begin{eqnarray*}
 g_{\hat Y_t}(\hat Y_t) & := &
 \mathbb E(\langle\mathbf D\hat Y_t,-\mathbf DL^{-1}\hat Y_t\rangle_{\mathcal H}|
 \hat Y_t)\\
 & = &
 \mathbb E(\langle\mathbf DY_t,-\mathbf DL^{-1}Y_t\rangle_{\mathcal H}|
 \hat Y_t) > 0.
\end{eqnarray*}
Therefore, by \cite{NV09}, Theorem 3.1 :
\begin{displaymath}
\mathbb P_{\hat Y_t}(dy) =
\frac{\mathbb E(|\hat Y_t|)}{2g_{\hat Y_t}(y)}
\exp\left[-\int_{0}^{y}\frac{z}{g_{\hat Y_t}(z)}dz\right]dy.
\end{displaymath}
Together with a straightforward application of the transfer theorem, that achieves the proof.
\end{proof}
%


%
\begin{corollary}\label{density_Jacobi}
Under assumptions \ref{assumption_beta}, \ref{auxiliary_assumption_cst} and \ref{Gaussian_assumption}, for every $x\in ]0,1[$,
\begin{displaymath}
\frac{d\mathbb P_{X_t}}{dx}(x) =
\frac{F'(x)\mathbb E(|\hat X_t|)}{2g_{F(X_t)}[F(x)]}
\exp\left[-\int_{\mathbb E[F(X_t)]}^{F(x)}
\frac{z -\mathbb E[F(X_t)]}{g_{F(X_t)}(z)}dz\right]
\end{displaymath}
where, $\hat X_t := F(X_t) -\mathbb E[F(X_t)]$.
\end{corollary}
%


%
\begin{proof}
Straightforward application of the transfer theorem which gives
\begin{displaymath}
\frac{d\mathbb P_{X_t}}{dx}(x) = f_t[F(x)]F'(x)
\textrm{ ; }
x\in ]0,1[,
\end{displaymath}
and of Proposition \ref{expression_of_the_density}.
\end{proof}
%


%
\section{A generalized Morris-Lecar neuron model}
\noindent
Let $X_t$ be a neuron's proportion of opened ion channels at time $t\geqslant 0$. In It\^o's calculus framework, as in Morris-Lecar's neuron model studied by S. Ditlevsen and P. Greenwood \cite{DG12}, if $X := (X_t,t\geqslant 0)$ is the solution of a Jacobi equation driven by a standard Brownian motion, it has $\alpha_0$-H\"older continuous paths on the compact intervals of $\mathbb R_+$ with $\alpha_0 < 1/2$. In particular, at time $t\geqslant 0$, there exists $h_0\in ]0,1[$ such that :
\begin{displaymath}
\forall h\in [0,h_0]
\textrm{, }
\left|X_{t + h} - X_t\right|
\leqslant
C_{t,h_0}h^{\alpha_0}
\leqslant 1.
\end{displaymath}
In other words, the increment of the proportion of opened ion channels between the times $t$ and $t + h$ is controlled by $\omega_{\alpha_0}(h) := C_{t,h_0}h^{\alpha_0}$.
\\
\\
Because of their various functions in the nervous system, there is an important morphological variability between neurons classes that implies an important variability of the number of ion channels, potentially opened, between two neurons belonging to different classes. For instance, Purkinje cells of the cerebellum receive $10^5$ synaptic inputs, and cortical pyramidal cells receive only 100 synaptic inputs (cf. L.F. Abbott and P. Dayan \cite{DA00} p. 2-3). Therefore, to control the increment of $X$ between $t$ and $t + h$ by $\omega_{\alpha_0}(h)$ could be too large for neurons having high total number of ion channels, and too small in the opposite case. For instance, in that second case ; on a very small time interval $[t,t + h]$, the increment of $X$ could be contained in $]\omega_{\alpha_0}(h),1[$. Then, one should replace $\omega_{\alpha_0}(h)$ by $\omega_{\alpha}(h) := \hat C_{t,h_0}h^{\alpha}$ with $\alpha <\alpha_0$. It means to assume that the process $X$ has $\alpha$-H\"older (but not $\alpha_0$-H\"older) continuous paths on the compact intervals of $\mathbb R_+$. For instance, the pathwise generalization of the Jacobi equation studied throughout this paper with a fractional Brownian signal $B^H$ of Hurst parameter $H\in ]0,1[$ works (with $\beta\in ]1 -\alpha,1[$ and $\alpha < H$).
\\
\\
On the deterministic Morris-Lecar model, the reader can refer to C. Morris and H. Lecar \cite{ML81}. The random dynamic of Morris-Lecar's model has been studied in T. Tateno and K. Pakdaman \cite{TP04}. On the stochastic Morris-Lecar model taken in the sense of It\^o for a standard Brownian signal and $\beta := 1/2$, please refer to S. Ditlevsen and P. Greenwood \cite{DG12}.
\\
\\
Let us now define a generalized Morris-Lecar neuron model :
\\
\\
Consider $V_t$ the membrane potential of the neuron and $X_t$ the normalized conductance of the $\textrm{K}^+$ current  (i.e. the probability that a $\textrm{K}^{+}$ ion channel is open) at time $t\geqslant 0$, and assume they satisfy the following equations in rough paths sense :
\begin{eqnarray}
 \label{Morris_Lecar_1}
 V_t & = & v_0 + \int_{0}^{t} b_V\left(V_s,X_s\right)ds
 \textrm{ and}\\
 \label{Morris_Lecar_2}
 X_t & = & x_0 + \int_{0}^{t} b_X\left(V_s,X_s\right)ds +
 \int_{0}^{t}\sigma_X\left(V_s,X_s\right)dB_{s}^{H}
\end{eqnarray}
where, $(v_0,x_0)\in I\times ]0,1[$ is a deterministic initial condition with $I := [-70\textrm{mV},30\textrm{mV}]$,
\begin{eqnarray*}
 b_V(v,x) & := & -\textrm{C}^{-1}\left[g_{\textrm{Ca}}m_{\infty}\left(V_t\right)\left(V_t - V_{\textrm{Ca}}\right) +
 g_{\textrm{K}}X_t\left(V_t - V_{\textrm{K}}\right) + g_{\textrm{L}}\left(V_t - V_{\textrm{L}}\right) - \textrm{I}\right],\\
 b_X(v,x) & := & a(v)(1 - x) - b(v)x,\\
 \sigma_X(v,x) & := & \sigma^*\left[2h(v)x(1-x)\right]^{\beta},\\
 m_{\infty}(v) & := & 1/2\left[1 + \tanh\left[\left(v - V_1\right)/V_2\right]\right],\\
 a(v) & := & 1/2\phi\cosh\left[\left(v - V_3\right)/\left(2V_4\right)\right]\left[1 + \tanh\left[\left(v - V_3\right)/V_4\right]\right],\\
 b(v) & := & 1/2\phi\cosh\left[\left(v - V_3\right)/\left(2V_4\right)\right]\left[1 - \tanh\left[\left(v - V_3\right)/V_4\right]\right],\\
 h(v) & := & a(v)b(v)/\left[a(v) + b(v)\right],
\end{eqnarray*}
$V_1$, $V_2$, $V_3$ and $V_4$ are scaling parameters, $g_{\textrm{Ca}}$ and $g_{\textrm{K}}$ are the maximal conductances associated to $\textrm{Ca}^{2+}$ and $\textrm{K}^+$, $g_{\textrm{L}}$ is the conductance associated to the leak current, $V_{\textrm{Ca}}$, $V_{\textrm{K}}$ and $V_{\textrm{L}}$ are the reversal potentials of $\textrm{Ca}^{2+}$, $\textrm{K}^+$ and leak currents respectively, $\textrm{C}$ is the membrane capacitance, $\phi$ is a rate scaling parameter, $\textrm{I}$ is the input current, $\sigma^*\in ]0,1[$ and $\beta$ satisfies Assumption \ref{assumption_beta}.
\\
\\
Equation (\ref{Morris_Lecar_2}) can be rewritten as a generalized Jacobi equation :
\begin{displaymath}
X_t =
x_0 -
\int_{0}^{t}\theta_s\left(X_s - \mu_s\right)ds +
\int_{0}^{t}\gamma_s\left[2\theta_sX_s\left(1 - X_s\right)\right]^{\beta}dB_{s}^{H}
\textrm{ ; }
t\geqslant 0
\end{displaymath}
where,
\begin{displaymath}
\mu_. :=
\frac{a\left(V_.\right)}{a\left(V_.\right) + b\left(V_.\right)}
\textrm{, }
\theta_. := a\left(V_.\right) + b\left(V_.\right)
\textrm{ and }
\gamma_{.}^{1/\beta} := (\sigma^*)^{1/\beta}
\frac{a\left(V_.\right)b\left(V_.\right)}{\left[a\left(V_.\right) + b\left(V_.\right)\right]^2}
\end{displaymath}
satisfy Assumption \ref{auxiliary_assumption}.
\\
\\
By Corollary \ref{existence_uniqueness_J_R_+}, the system (\ref{Morris_Lecar_1})-(\ref{Morris_Lecar_2}) admits a unique bounded solution.
%


%
\appendix
\section{Probabilistic preliminaries}
\noindent
This appendix is devoted to state some results and notations, used throughout the paper, on random dynamical systems (cf. L. Arnold \cite{ARNOLD98}) and Malliavin calculus (cf. D. Nualart \cite{NUALART06}).
%


%
\subsection{Random dynamical systems}
Inspired by L. Arnold \cite{ARNOLD98}, this subsection provides some definitions and results on random dynamical systems.
%


%
\begin{definition}\label{DS}
A family $\vartheta := (\theta_t,t\in\mathbb R)$ of maps from a measurable space $(\Omega,\mathcal A)$ into itself is a (measurable) dynamical system (DS) if and only if :
\begin{enumerate}
 \item $(t,\omega)\mapsto\theta_t\omega$ is $\mathcal B(\mathbb R)\otimes\mathcal A,\mathcal A$-measurable.
 \item $\theta_0 = \textrm{Id}_{\Omega}$.
 \item For every $s,t\in\mathbb R$, $\theta_{s + t} = \theta_s\circ\theta_t$ ((semi-)flow property).
\end{enumerate}
A measure $\mu$ on $(\Omega,\mathcal A)$ is $\vartheta$-invariant if and only if $\theta_t\mu = \mu$ for every $t\in\mathbb R$, where $\theta_t\mu := \mu(\theta_t\in .)$.
\end{definition}
%


%
\begin{definition}\label{invariant_mod_P_sets}
Consider a probability space $(\Omega,\mathcal A,\mathbb P)$ and a DS $\vartheta$ on $(\Omega,\mathcal A)$. A set $A\in\mathcal A$ is invariant $\mod\mathbb P$ with respect to $\vartheta$ if and only if,
\begin{displaymath}
\forall
t\in\mathbb R
\textrm{$,$ }
\mathbb P\left(A\Delta\{\theta_t\in A\}\right) = 0
\end{displaymath}
where, $E\Delta F := \left(E^c\cap F\right)\cup\left(E\cap F^c\right)\textrm{ $;$ }\forall E,F\in\mathcal A$.
\end{definition}
\noindent
\textbf{Notations :}
\begin{enumerate}
 \item $I_{\mathbb P}$ is the $\sigma$-algebra of sets invariant $\mod\mathbb P$.
 \item In the sequel, $\mathbb E$ denotes the expectation for the probability measure $\mathbb P$.
\end{enumerate}
%


%
\begin{theorem}\label{ergodic_theorem}
(Birkhoff-Chintchin) Consider a probability space $(\Omega,\mathcal A,\mathbb P)$ and a DS $\vartheta$ on $(\Omega,\mathcal A)$. If $\mathbb P$ is $\vartheta$-invariant, for every $f\in L^1(\Omega;\mathbb P)$,
\begin{displaymath}
\lim_{t\rightarrow\infty}
\frac{1}{t}
\int_{0}^{t}
f(\theta_s .)ds =
\mathbb E(f | I_{\mathbb P})
\textrm{ $\mathbb P$-a.s.}
\end{displaymath}
\end{theorem}
\noindent
\textbf{Remark.} The metric dynamical system (metric DS) $(\Omega,\mathcal A,\mathbb P,\vartheta)$ is ergodic if and only if every sets belonging to $I_{\mathbb P}$ have probability $0$ or $1$. In that case, with notations of Theorem \ref{ergodic_theorem}, $\mathbb E(f|I_{\mathbb P}) = \mathbb E(f)$.
%


%
\begin{definition}\label{RDS}
Consider a metric space $X$ and a metric DS $(\Omega,\mathcal A,\mathbb P,\vartheta)$. A random dynamical system (RDS) on the measurable space $(X,\mathcal B(X))$ over the metric DS $(\Omega,\mathcal A,\mathbb P,\vartheta)$ is a map $\varphi :\mathbb R_+\times\Omega\times X\rightarrow X$ such that :
\begin{enumerate}
 \item $\varphi$ is $\mathcal B(\mathbb R_+)\otimes\mathcal A\otimes\mathcal B(X),\mathcal B(X)$-measurable (measurability).
 \item For every $s,t\in\mathbb R_+$ and $\omega\in\Omega$,
 \begin{enumerate}
  \item $\varphi(0,\omega) = \textrm{Id}_X$,
  \item $\varphi(s + t,\omega) = \varphi(t,\theta_s\omega)\circ\varphi(s,\omega)$
 \end{enumerate}
 (cocycle property).
\end{enumerate}
If for every $\omega\in\Omega$, $(t,x)\mapsto\varphi(t,\omega)x$ is continuous from $\mathbb R_+\times X$ into $X$, then $\varphi$ is a continuous random dynamical system.
\end{definition}
%


%
\begin{proposition}\label{skew_product}
Consider a metric space $X$, a metric DS $(\Omega,\mathcal A,\mathbb P,\vartheta)$, a random dynamical system $\varphi$ on the measurable space $(X,\mathcal B(X))$ over the metric DS $(\Omega,\mathcal A,\mathbb P,\vartheta)$, and the family $\Theta := (\Theta_t,t\in\mathbb R_+)$ of maps from $\Omega\times X$ into itself such that :
\begin{displaymath}
\Theta_t(\omega,x) :=
(\theta_t\omega,\varphi(t,\omega)x)
\textrm{ $;$ }
\forall t\in\mathbb R_+
\textrm{$,$ }
\forall\omega\in\Omega
\textrm{$,$ }
\forall x\in X.
\end{displaymath}
Then $\Theta$ defines a dynamical system, called skew product of the metric DS $(\Omega,\mathcal A,\mathbb P,\vartheta)$ and the RDS $\varphi$ on $X$.
\end{proposition}
\noindent
\textbf{Notation.} Consider a metric space $X$, a metric DS $(\Omega,\mathcal A,\mathbb P,\vartheta)$ and the map $p_{\Omega} :\Omega\times X\rightarrow\Omega$ defined by :
\begin{displaymath}
p_{\Omega}(\omega,x) :=\omega
\textrm{ $;$ }\forall (\omega,x)\in\Omega\times X.
\end{displaymath}
For every probability measure $\mu$ on $(\Omega\times X,\mathcal A\otimes\mathcal B(X))$, $p_{\Omega}\mu :=\mu(p_{\Omega}\in .)$.
\\
\\
With the notations of Proposition \ref{skew_product}, since $\Theta$ is a dynamical system depending on the metric DS $(\Omega,\mathcal A,\mathbb P,\vartheta)$ which is "given" to us from outside and cannot be manipulated in L. Arnold's philosophy, and on the RDS $\varphi$, it is natural to assume that a $\varphi$-invariant probability measure is a $\Theta$-invariant measure such that its marginal $p_{\Omega}\mu$ on $(\Omega,\mathcal A)$ coincides with $\mathbb P$.
%


%
\begin{definition}\label{invariant_measures_RDS}
Consider a metric space $X$, a metric DS $(\Omega,\mathcal A,\mathbb P,\vartheta)$, a random dynamical system $\varphi$ on the measurable space $(X,\mathcal B(X))$ over the metric DS $(\Omega,\mathcal A,\mathbb P,\vartheta)$, and $\Theta$ the skew product of the metric DS $(\Omega,\mathcal A,\mathbb P,\vartheta)$ and the RDS $\varphi$ on $X$. A probability measure $\mu$ on $(\Omega\times X,\mathcal A\otimes\mathcal B(X))$ is $\varphi$-invariant if and only if $\mu$ satisfies the two following properties :
\begin{enumerate}
 \item $\mu$ is $\Theta$-invariant.
 \item $p_{\Omega}\mu =\mathbb P$.
\end{enumerate}
\end{definition}
\noindent
\textbf{Notations :}
\begin{itemize}
 \item $\mathcal P_{\mathbb P}(\Omega\times X) :=\{\mu\textrm{ probability measure on }(\Omega\times X,\mathcal A\otimes\mathcal B(X)) : p_{\Omega}\mu =\mathbb P\}$.
 \item $\mathcal I_{\mathbb P}(\varphi)$ denotes the set of $\varphi$-invariant probability measures.
\end{itemize}
%


%
\begin{proposition}\label{factorization}
Consider a metric space $X$, a metric DS $(\Omega,\mathcal A,\mathbb P,\vartheta)$, a random dynamical system $\varphi$ on the measurable space $(X,\mathcal B(X))$ over the metric DS $(\Omega,\mathcal A,\mathbb P,\vartheta)$, and $\Theta$ the skew product of the metric DS $(\Omega,\mathcal A,\mathbb P,\vartheta)$ and the RDS $\varphi$ on $X$. For each $\mu\in\mathcal P_{\mathbb P}(\Omega\times X)$, there exists a unique family $(\mu_{\omega},\omega\in\Omega)$ of measures on $(X,\mathcal B(X))$ such that
\begin{displaymath}
\mu(d\omega,dx) =
\mu_{\omega}(dx)\mathbb P(d\omega)
\textrm{ $\mathbb P$-a.s.}
\end{displaymath}
where,
\begin{enumerate}
 \item For every $B\in\mathcal B(X)$, $\omega\in\Omega\mapsto\mu_{\omega}(B)$ is $\mathcal A$-measurable.
 \item For $\mathbb P$-almost every $\omega\in\Omega$, $\mu_{\omega}$ is a probability measure on $(X,\mathcal B(X))$.
\end{enumerate}
That family is the factorization of $\mu$ with respect to $\mathbb P$.
\end{proposition}
\noindent
For a proof, please refer to \cite{ARNOLD98}, Proposition 1.4.3.
%


%
\begin{theorem}\label{invariant_measure_compact}
Consider a compact metric space $X$, a metric DS $(\Omega,\mathcal A,\mathbb P,\vartheta)$, a continuous random dynamical system $\varphi$ on the measurable space $(X,\mathcal B(X))$ over the metric DS $(\Omega,\mathcal A,\mathbb P,\vartheta)$, and $\Theta$ the skew product of the metric DS $(\Omega,\mathcal A,\mathbb P,\vartheta)$ and the RDS $\varphi$ on $X$. Then, the convex compact set of $\varphi$-invariant probability measures $\mathcal I_{\mathbb P}(\varphi)$ is non-void.
\end{theorem}
\noindent
For a proof, please refer to \cite{ARNOLD98}, Theorem 1.5.10.
%


%
\subsection{Malliavin calculus}
Essentially inspired by D. Nualart \cite{NUALART06}, this subsection provides some definitions and results on Malliavin calculus. Moreover, basics on the Ornstein-Uhlenbeck semi-group and operator are stated in the last part.
\\
\\
Let $W$ be a $1$-dimensional centered Gaussian process with continuous paths on $[0,T]$ ($T > 0$). Its canonical probability space is denoted by $(\Omega,\mathcal A,\mathbb P)$. First, let us introduce two Hilbert spaces associated to that process :
\\
\\
On the one hand, the Cameron-Martin's space of $W$ is given by
\begin{displaymath}
\mathcal H^1 :=
\left\{
h\in C^0([0,T];\mathbb R) :
\exists Z\in\mathcal W
\textrm{ s.t. }
\forall t\in [0,T]\textrm{, }
h_t =
\mathbb E(W_tZ)
\right\}
\end{displaymath}
with
\begin{displaymath}
\mathcal W :=
\overline{\textrm{span}
\left\{W_t,
t\in [0,T]\right\}}^{L^2}.
\end{displaymath}
Let $\langle .,.\rangle_{\mathcal H^1}$ be the map defined on $\mathcal H^1\times\mathcal H^1$ by
\begin{displaymath}
\langle h,\hat h\rangle_{\mathcal H^1} :=
\mathbb E(Z\hat Z)
\end{displaymath}
where,
\begin{displaymath}
\forall t\in [0,T]\textrm{, }
h_t =
\mathbb E(W_tZ)
\textrm{ and }
\hat h_t =
\mathbb E(W_t\hat Z)
\end{displaymath}
for every $Z,\hat Z\in\mathcal W$.
\\
\\
That map is a scalar product on $\mathcal H^1$ and, equipped with it, $\mathcal H^1$ is a Hilbert space.
\\
\\
On the other hand, consider the set $\mathcal E$ of functions defined on $[0,T]$ by
\begin{displaymath}
\sum_{k = 1}^{n}
a_k\mathbf 1_{[0,s_k]}
\textrm{ ; }
n\in\mathbb N^*
\textrm{, }
(s_1,\dots,s_n)\in [0,T]^n
\textrm{, }
(a_1,\dots,a_n)\in\mathbb R^n,
\end{displaymath}
and $\mathcal H$ the closure of $\mathcal E$ for the scalar product $\langle .,.\rangle_{\mathcal H}$ defined by
\begin{displaymath}
\langle
\sum_{k = 1}^{n}a_k\mathbf 1_{[0,s_k]};
\sum_{l = 1}^{m}b_l\mathbf 1_{[0,t_l]}\rangle_{\mathcal H} :=
\sum_{k = 1}^{n}
\sum_{l = 1}^{m}a_kb_l\mathbb E\left(W_{s_k}W_{t_l}\right)
\end{displaymath}
for every $n,m\in\mathbb N^*$, $(s_1,\dots,s_n)\in [0,T]^n$, $(t_1,\dots,t_m)\in [0,T]^m$, $(a_1,\dots,a_n)\in\mathbb R^n$ and $(b_1,\dots,b_m)\in\mathbb R^m$. Equipped with the scalar product $\langle .,.\rangle_{\mathcal H}$, $\mathcal H$ is a separable Hilbert space called reproducing kernel Hilbert space of $W$ (cf. J. Neuveu \cite{NEUVEU68}).
\\
\\
Let $\mathbf W$ be the map defined on $\mathcal E$ by
\begin{displaymath}
\mathbf W\left(
\sum_{k = 1}^{n}a_k\mathbf 1_{[0,s_k]}\right) :=
\sum_{k = 1}^{n}a_kW_{s_k}
\end{displaymath}
for every $n\in\mathbb N^*$, $(s_1,\dots,s_n)\in [0,T]^n$ and $(a_1,\dots,a_n)\in\mathbb R^n$. It extends to $\mathcal H$ as a map called Wiener integral with respect to $W$, and $\{\mathbf W(h),h\in\mathcal H\}$ is an iso-normal Gaussian process (cf. \cite{NUALART06}, Definition 1.1.1).
\\
\\
Let $I$ be the map from $\mathcal H$ into $\mathcal H^1$ defined by
\begin{displaymath}
I(h) :=
\mathbb E\left[\mathbf W(h)W\right]\in\mathcal H^1
\end{displaymath}
for every $h\in\mathcal H$. It is an isometry from $\mathcal H$ into $\mathcal H^1$.
\\
\\
\textbf{Example (fractional Brownian motion).} A centered Gaussian process $B^H$ is a fractional Brownian motion of Hurst parameter $H\in ]0,1[$ if its covariance function is defined by :
\begin{displaymath}
R_H(t,s) =
\frac{1}{2}(s^{2H} + t^{2H} - |t - s|^{2H})
\textrm{ ; }
\forall s,t\in [0,T].
\end{displaymath}
The scalar product $\langle .,.\rangle_{\mathcal H}$ is explicit (cf. L. Decreusefond and A.S. Ust\"unel \cite{DU99} and D. Nualart \cite{NUALART06}, Subsection 5.1.3) :
\begin{itemize}
 \item If $H = 1/2$, $\mathcal H = L^2([0,T])$ and $\langle .,.\rangle_{\mathcal H}$ is the usual scalar product \mbox{on $L^2([0,T])$.}
 \item If $H > 1/2$,
 \begin{displaymath}
 \langle\varphi,\psi\rangle_{\mathcal H} =
 \alpha_H\int_{0}^{T}\int_{0}^{T}
 |t - s|^{2H - 2}\varphi(s)\psi(t)dsdt
 \textrm{ ; }
 \forall
 \varphi,\psi\in\mathcal H
 \end{displaymath}
 with $\alpha_H := H(2H - 1)$.
 \item If $H < 1/2$, let $K_H$ be the function defined on $\Delta_T$ by
 \begin{eqnarray*}
  K_H(t,s) & := &
  c_H\left[
  \left(\frac{t}{s}\right)^{H - 1/2}(t - s)^{H - 1/2}\right. -\\
  & & \left.\left(H -\frac{1}{2}\right)
  s^{1/2 - H}\int_{s}^{t}u^{H - 3/2}(u - s)^{H - 1/2}du\right]
  \textrm{ ; }
  \forall (s,t)\in\Delta_T
 \end{eqnarray*}
 where, $c_H > 0$ denotes a deterministic constant only depending on $H$. Let also $K_{H}^{*} :\mathcal H\rightarrow L^2([0,T])$ be the map defined by
 \begin{displaymath}
 (K_{H}^{*}\varphi)(s) :=
 K_H(T,s)\varphi(s) +
 \int_{s}^{T}
 \left[\varphi(t) -\varphi(s)\right]
 \frac{\partial K_H}{\partial t}(t,s)dt
 \textrm{ ; }
 \forall s\in [0,T]
 \textrm{, }
 \forall
 \varphi\in\mathcal H.
 \end{displaymath}
 Then,
 \begin{displaymath}
 \langle\varphi,\psi\rangle_{\mathcal H} =
 \int_{0}^{T}
 (K_{H}^{*}\varphi)(s)
 (K_{H}^{*}\psi)(s)ds
 \textrm{ ; }
 \forall
 \varphi,\psi\in\mathcal H.
 \end{displaymath}
\end{itemize}
Now, let us provide basics on Malliavin calculus for the isonormal Gaussian process $\mathbf W$ defined above.
\\
\\
\textbf{Notation.} The set of all continuously differentiable functions $f :\mathbb R^n\rightarrow\mathbb R$ such that $F$ and all its partial derivatives have polynomial growth is denoted by $C_{P}^{\infty}(\mathbb R^n;\mathbb R)$.
%


%
\begin{definition}\label{Malliavin_derivative}
The Malliavin derivative of a smooth functional
\begin{displaymath}
F = f\left[
\mathbf W\left(h_1\right),\dots,
\mathbf W\left(h_n\right)\right]
\end{displaymath}
where $n\in\mathbb N^*$, $f\in C_{P}^{\infty}(\mathbb R^n;\mathbb R)$ and $h_1,\dots,h_n\in\mathcal H$ is the following $\mathcal H$-valued random variable :
\begin{displaymath}
\mathbf DF :=
\sum_{k = 1}^{n}
\partial_k f
\left[
\mathbf W\left(h_1\right),\dots,
\mathbf W\left(h_n\right)\right]h_k.
\end{displaymath}
\end{definition}
%


%
\begin{proposition}\label{Malliavin_derivative_domain}
The map $\mathbf D$ is closable from $L^p(\Omega;\mathbb P)$ into $L^p(\Omega,\mathcal H;\mathbb P)$ for every $p\geqslant 1$. The domain of $\mathbf D$ in $L^p(\Omega;\mathbb P)$ is denoted by $\mathbb D^{1,p}$. It is the closure of the smooth functionals space for the norm $\|.\|_{1,p}$ such that
\begin{displaymath}
\|F\|_{1,p}^{p} :=
\mathbb E(|F|^p) +
\mathbb E(\|\mathbf DF\|_{\mathcal H}^{p}) < \infty
\end{displaymath}
for every $F\in L^p(\Omega;\mathbb P)$
\end{proposition}
\noindent
For a proof, refer to \cite{NUALART06}, Proposition 1.2.1.
%


%
\begin{definition}\label{Malliavin_derivative_loc}
A random variable $F$ is locally derivable in the sense of Malliavin if and only if there exists a sequence $((\Omega_n,F_n) ; n\in\mathbb N^*)$ of elements of $\mathcal A\times\mathbb D^{1,2}$ such that $\Omega_n\uparrow\Omega$ when $n\rightarrow\infty$ and, $F = F_n$ on $\Omega_n$ for every $n\in\mathbb N^*$. Such random variables define a vector space denoted by $\mathbb D_{\textrm{loc}}^{1,2}$, and containing $\mathbb D^{1,2}$.
\end{definition}
%


%
\begin{definition}\label{H1_differentiability}
A random variable $F : \Omega\rightarrow\mathbb R$ is continuously $\mathcal H^1$-differentiable if and only if, for almost every $\omega\in\Omega$, $h\mapsto F(\omega + h)$ is continuously differentiable from $\mathcal H^1$ into $\mathbb R$.
\end{definition}
%


%
\begin{proposition}\label{link_H1_differentiability_Malliavin_derivability}
A continuously $\mathcal H^1$-differentiable random variable $F :\Omega\rightarrow\mathbb R$ is locally derivable in the sense of Malliavin. Moreover, if $\mathbb E(F^2) < \infty$ and $\mathbb E(\|\mathbf DF\|_{\mathcal H}^{2}) < \infty$, then $F\in\mathbb D^{1,2}$ and, for almost every $\omega\in\Omega$ and every $h\in\mathcal H^1$,
\begin{displaymath}
\langle\mathbf DF(\omega),I^{-1}(h)\rangle_{\mathcal H} =
D_hF^{\omega}(0)
\end{displaymath}
with $F^{\omega} = F(\omega + .)$.
\end{proposition}
\noindent
For proofs, refer to \cite{NUALART06}, Proposition 4.1.3 and Lemma 4.1.2.
\\
\\
\textbf{Remark.} For the sake of simplicity, $D_hF^{.}(0)$ is denoted by $D_hF$.
\\
\\
The last part of this subsection is devoted to the Ornstein-Uhlenbeck semi-group and operator :
%


%
\begin{definition}\label{OU_operator}
The Ornstein-Uhlenbeck semi-group $(T_u, u\in\mathbb R_+)$ is the family of operators defined on $L^2(\Omega;\mathbb P)$ by
\begin{displaymath}
T_uF :=
\sum_{n = 0}^{\infty}
e^{-nu}J_nF
\textrm{ $;$ }
\forall
u\in\mathbb R_+
\end{displaymath}
and its generator, the Ornstein-Uhlenbeck operator $L$, satisfies
\begin{displaymath}
LF =
-\sum_{n = 0}^{\infty}
nJ_nF
\end{displaymath}
for any $F\in L^2(\Omega;\mathbb P)$, of chaos expansion
\begin{displaymath}
F =
\sum_{n = 0}^{\infty}
J_nF.
\end{displaymath}
\end{definition}
\noindent
\textbf{Remarks :}
\begin{enumerate}
 \item About the chaos expansion of square integrable random variables, see \cite{NUALART06}, Section 1.1. About the Ornstein-Uhlenbeck semi-group and operator, see \cite{NUALART06}, Section 1.4.
 \item The operator $T_u$ is nonnegative for every $u\in\mathbb R_+$ :
 \begin{displaymath}
 \forall F\in L^2(\Omega;\mathbb P)
 \textrm{, }
 F\geqslant 0
 \Longrightarrow
 T_uF\geqslant 0.
 \end{displaymath}
 \item The operator $L$ is invertible on the subspace of centered random variables belonging to $L^2(\Omega;\mathbb P)$, and
 \begin{displaymath}
 L^{-1}F =
 -\sum_{n = 0}^{\infty}\frac{1}{n}J_nF.
 \end{displaymath}
 If $\mathbb E(F)\not= 0$, $L^{-1}F$ still exists, and $LL^{-1}F = F -\mathbb E(F)$.
\end{enumerate}
%


%

%
\end{document}